\documentclass[11pt]{amsart}

\title{Odd Degree Isolated Points on $X_1(N)$ with Rational $j$-invariant}

\author{Abbey Bourdon}
\address{Wake Forest University, Department of Mathematics and Statistics, Winston-Salem, NC 27109, USA}
\email{bourdoam@wfu.edu}
\urladdr{http://users.wfu.edu/bourdoam/}

\author{David R. Gill}
\email{gilldr19@wfu.edu}

\author{Jeremy Rouse}
\email{rouseja@wfu.edu}
\urladdr{http://users.wfu.edu/rouseja/}

\author{Lori D. Watson}
\email{watsonl@wfu.edu}
\urladdr{https://loridwatson.com}

\usepackage{amsmath,amssymb,amsthm,sistyle,mathtools,enumitem,textcomp,url,hyperref}
\usepackage[all]{xy}
\usepackage{geometry}
\usepackage{tikz}
\SIthousandsep{,}
\usepackage{fullpage}
\usepackage{color}
\DeclareMathAlphabet{\curly}{U}{rsfs}{m}{n}
\newtheorem{thm}{Theorem}[]

\newtheorem{cor}[thm]{Corollary}
\newtheorem{prop}[thm]{Proposition}
\newtheorem{lem}[thm]{Lemma}

\theoremstyle{definition}

\theoremstyle{remark}
\newtheorem{rmk}[thm]{Remark}

\mathchardef\mhyphen="2D

\begin{document}
\newcommand\leg{\genfrac(){.4pt}{}}
\renewcommand{\labelenumi}{(\roman{enumi})}
\def\ff{\mathfrak{f}}
\def\F{\mathbb{F}}
\def\N{\mathbb{N}}
\def\Q{\mathbb{Q}}
\def\Z{\mathbb{Z}}
\def\R{\mathbb{R}}
\def\OO{\mathcal{O}}
\def\aa{\mathfrak{a}}
\def\pp{\mathfrak{p}}
\def\qq{\mathfrak{q}}
\def\Aa{\curly{A}}
\def\Dd{\curly{D}}
\def\Gg{\curly{G}}
\def\Pp{\curly{P}}
\def\Ss{\curly{S}}
\def\End{\mathrm{End}}
\def\M{\curly{M}}
\newcommand{\Ok}{\mathcal{O}_K}
\newcommand{\tors}{\operatorname{tors}}
\newcommand{\Pic}{\operatorname{Pic}}
\newcommand{\ord}{\operatorname{ord}}
\newcommand{\GL}{\operatorname{GL}}
\newcommand{\Gal}{\operatorname{Gal}}
\newcommand{\gcdl}{\operatorname{gcd}}
\newcommand{\im}{\operatorname{im}}
\newcommand{\Supp}{\operatorname{Supp}}
\newcommand{\Aut}{\operatorname{Aut}}
\newcommand{\gon}{\operatorname{gon}}
\newcommand{\proj}{\operatorname{proj}}

\begin{abstract}
Let $C$ be a curve defined over a number field $k$. We say a closed point $x\in C$ of degree $d$ is isolated if it does not belong to an infinite family of degree $d$ points parametrized by the projective line or a positive rank abelian subvariety of the curve's Jacobian. Building on work of \cite{BELOV}, we characterize elliptic curves with rational $j$-invariant which give rise to an isolated point of odd degree on $X_1(N)/\Q$ for some positive integer $N$. 
\end{abstract}

\maketitle
\section{Introduction}
Let $C$ be a curve defined over a number field $k$, and let $x \in C$ be a closed point of degree $d$. Following \cite{BELOV}, we say $x$ is \textbf{isolated} if it does not belong to an infinite family of degree $d$ points parametrized by $\mathbb{P}^1$ or a positive rank abelian subvariety of the curve's Jacobian. (See $\S2.5$ for details.) Motivated by the well-known problem of classifying torsion subgroups of elliptic curves over number fields, we seek to describe isolated points on the modular curve $X_1(N)/\Q$. As a first case, we focus on those isolated points corresponding to elliptic curves with rational $j$-invariant. That is, we consider isolated points $x \in X_1(N)$ such that $j(x) \in \Q$, where $j:X_1(N) \rightarrow X_1(1) \cong \mathbb{P}^1$ denotes the $j$-map. Though there are infinitely many isolated points with this condition---indeed, there are infinitely many isolated points above any $j$-invariant associated to an elliptic curve with complex multiplication (CM) by \cite[Thm. 7.1]{BELOV}---there is strong evidence that all isolated points $x \in X_1(N)$ with $j(x) \in \Q$ arise from points on one of a \emph{finite} number of elliptic curves, even as $N$ ranges over all positive integers.

\begin{thm}[Bourdon, Ejder, Liu, Odumodu, Viray \cite{BELOV}] \label{BELOVthm}
Let $\mathcal{I}$ denote the set of all isolated points on all modular curves $X_1(N)$ for $N \in \mathbb{Z}^+$. Suppose there exists a constant $C=C(\Q)$ such that for all non-CM elliptic curves $E/\Q$, the mod $p$ Galois representation associated to $E$ is surjective for primes $p>C$. Then $j(\mathcal{I}) \cap \Q$ is finite.
\end{thm}
\noindent The existence of a constant $C$ as in the theorem statement was first suggested in a question of Serre \cite{serre72}, and in \cite{serre81} he asked whether $C(\Q)=37$. Significant partial results combined with computational evidence have led to this increasingly standard assumption becoming known as Serre's Uniformity Conjecture. See for example \cite{BP11}, \cite{BPR13}, \cite{Balakrishnan}, \cite{ZywinaImages}, \cite{sutherland}, \cite{lemosTrans}, \cite{lemosZ}.

A natural problem in light of Theorem \ref{BELOVthm} is to identify the (likely finite) set $j(\mathcal{I}) \cap \Q$. By \cite[Thm. 7.1]{BELOV}, the set contains all 13 CM $j$-invariants in $\Q$ as well as at least two non-CM $j$-invariants: $-3^2\cdot 5^6/2^3$, corresponding {to two isolated points} of degree 3 on $X_1(21)$ identified by Najman \cite{najman16}, and $-7\cdot 11^3$, corresponding to degree 6 points on $X_1(37)$ lying above one of the two non-cuspidal rational points on $X_0(37)$. Here, we give an unconditional version of Theorem \ref{BELOVthm} by restricting our attention to points of odd degree. Our main result is the following:

{
\begin{thm} \label{MainThm}
Let $\mathcal{I}_{odd}$ denote the set of all isolated points of odd degree on all modular curves $X_1(N)$ for $N \in \mathbb{Z}^+$. Then $j(\mathcal{I}_{odd}) \cap \Q$ contains at most the $j$-invariants in the following list:
\begin{center}
\begin{tabular}{c|c}
non-CM $j$-invariants              & CM $j$-invariants       \\ \hline
$-3^2\cdot 5^6/2^3$                &    $-2^{18}\cdot3^3\cdot5^3$   \\ \hline
     $3^3\cdot 13/2^2$            &   $-2^{15}\cdot3^3\cdot5^3\cdot11^3$        \\ \hline
 &  $-2^{18}\cdot3^3\cdot5^3\cdot23^3\cdot29^3$  \\                     
\end{tabular}
\end{center}
Conversely, $j(\mathcal{I}_{odd}) \cap \Q$ contains $-3^2\cdot 5^6/2^3$ and $3^3\cdot 13/2^2$.
\end{thm}}

{The $j$-invariant $3^3\cdot 13/2^2$ corresponds to a degree 9 point on $X_1(28)$. The existence of such a point was noted during an extensive computational search performed by Najman and Gonz\'{a}lez-Jim\'{e}nez (see \cite{GJNalgorithm}), and in fact it can be realized by a rational elliptic curve under base extension. However, this is the first instance this point has been identified as isolated.} The {CM} $j$-invariants give points of degree 21 on $X_1(43)$, degree 33 on $X_1(67)$, and degree 81 on $X_1(163)$, respectively. They are in $j(\mathcal{I}_{odd}) \cap \Q$ if and only if these points are isolated. One difficulty in determining whether they are in fact isolated stems from the fact that the {Jacobian variety of each} of the last three curves has positive rank; see Remark \ref{finalCMjRmk}. 

The first step in the proof of Theorem \ref{MainThm} is to establish a connection between points on $X_1(N)$ of odd degree and rational isogenies. This is analogous to the connection found in the case of odd degree CM points on modular curves \cite[Cor. 9.4]{aoki95}, and it relies on the classification of rational isogenies of elliptic curves over $\Q$ due to Mazur \cite{mazur}, Kenku \cite{kenku}, and others.
\begin{thm}\label{OddDegThm}
Let $x\in X_1(n)$ be a point of odd degree with $j(x) \in \Q$. If $j(x) \neq 3^3\cdot 5\cdot 7^5/2^7$ and $p$ is an odd prime dividing $n$, then there exists $y \in X_0(p)(\Q)$ with $j(x)=j(y)$. Moreover:
\begin{enumerate}
\item If {$j(x) \neq j(z)$ for all $z \in X_0(21)(\Q)$}, then $n=2^ap^b$ for $p \in \{3,5,7,11,13, 19, 43, 67, 163\}$ and nonnegative integers $a,b$ with $a \leq 3$. If $b>0$, then $a \leq 2$
\item If {there exists $z \in X_0(21)(\Q)$ with $j(x)=j(z)$,} then $n=2^a3^b7^c$ for nonnegative integers $a,b,c$ with $a \leq 1$.
\end{enumerate} 
If $j(x)=3^3\cdot 5\cdot 7^5/2^7$, then $n=2^a7^b$ for nonnegative integers $a,b$ with $a \leq 1$.
\end{thm}

\begin{rmk} From work of Zywina \cite{ZywinaImages}, it is known that the mod 7 Galois representation of a non-CM elliptic curve over $\Q$ is properly contained in the normalizer of a split Cartan subgroup if and only if $j(E)=3^3\cdot 5\cdot 7^5/2^7$.
\end{rmk}
\begin{rmk}
There are precisely 4 non-cuspidal points in $X_0(21)(\Q)$ which correspond to non-CM elliptic curves with $j$-invariants $-3^2\cdot 5^6/2^3$, $3^3\cdot5^3/2$, $-3^2\cdot 5^3 \cdot 101^3/2^{21}$, $-3^3\cdot 5^3 \cdot 383^3/2^7$. As mentioned above, the first of these is known to correspond to an isolated point.
\end{rmk}

Provided $j(x) \neq 3^3\cdot 5\cdot 7^5/2^7$ and does not correspond to a CM elliptic curve, we can deduce information about the degree of $x \in X_1(2^ap^b)$ using work of Greenberg \cite{greenberg2012} and Greenberg, Rubin, Silverberg, and Stoll \cite{GRSS14} which concerns the image of Galois representations of an elliptic curve over $\Q$ with a rational isogeny. Often, the degree of $x$ is as large as possible given the degree of its image in $X_1(2^ap)$ or $X_1(2^ap^2)$, which means an isolated point would remain isolated under the natural projection map \cite{BELOV}. We must then determine whether isolated points corresponding to elliptic curves with rational $j$-invariant exist on this curve of lower level. If $p=3$, we rely on the classification of 3-adic images of elliptic curves over $\Q$ due to Rouse, Sutherland, and Zureick-Brown \cite{3adicimages}. Our proof involves finding all rational points on an explicit genus 4 curve which characterizes a certain kind of ``entanglement" of torsion point fields; see Proposition \ref{54prop}. {Other notable cases include elliptic curves with rational cyclic isogenies of degree 21 or 25. In the first case, we show in $\S4.2$ that known instances of entanglement can be leveraged to produce bounds on the level of certain Galois representations which improve those given in \cite[Prop. 6.1]{BELOV}. For elliptic curves with a rational cyclic 25-isogeny, our arguments use intermediate modular curves lying between $X_1(N)$ and $X_0(N)$ in addition to refined results of Greenberg \cite{greenberg2012}; see Proposition \ref{25lem}. }

The results on CM elliptic curves follow from work of Kwon \cite{kwon99}, Aoki \cite{aoki95}, and recent work of the first author and Pete L. Clark \cite{BC1}, \cite{BC2}. 

\section*{Acknowledgements}
We thank Pete L. Clark for helpful conversations. We also thank Filip Najman, Bianca Viray, and David Zureick-Brown for helpful comments on an earlier draft. The first author was partially supported by an A. J. Sterge Faculty Fellowship.

\section{Background and Notation}

\subsection{Galois representations of elliptic curves.}
Let $k$ be a number field and let $E/k$ be an elliptic curve. Then for any fixed $N \in \Z^+$ the points of $E(\overline{k})$ with order dividing $N$, denoted $E[N]$, form a free $\Z/N\Z$-module of rank 2. By choosing a basis for $E[N]$, the action of the absolute Galois group of $k$, denoted $\Gal_k$, is recorded in the mod $N$ Galois representation associated to $E$
\[
\rho_{E,N}: \Gal_k \rightarrow \Aut(E[N]) \cong \GL_2(\Z/N\Z).
\]
Taking the inverse limit over all $N$, we obtain the adelic Galois representation associated to $E$, which gives the Galois action on all torsion points of $E$
\[
\rho_{E}: \Gal_k \rightarrow  \Aut(E(\overline{k})_{\tors}) \cong  \GL_2(\widehat{\Z}).
\]
For any positive integer $m$, we may compose $\rho_E$ with projection onto the restricted product
\[
\rho_{E,m^{\infty}}: \Gal_k \xrightarrow{\rho_E} \GL_2(\widehat{\Z}) \cong  \prod_{p \text{ prime}} \GL_2(\Z_{p}) \xrightarrow{\proj} \prod_{p \mid m} \GL_2(\Z_{p}),
\]
obtaining the $m$-adic representation associated to $E$. {More generally, if $m,n$ are relatively prime positive integers, we write $\rho_{E,m \cdot n^{\infty}}$ for $\rho_E$ composed with the natural projection
\[
\GL_2(\widehat{\Z}) \cong \prod_{p \text{ prime}}\GL_2(\Z_p) \rightarrow \GL_2(\Z/m\Z) \times \prod_{p \mid n} \GL_2(\Z_{p}).
\] Throughout we use $\pi$ to denote the natural reduction map.}

For a fixed non-CM elliptic curve $E/k$, Serre's Open Image Theorem \cite{serre72} states that $\im \rho_E$ is open in $\GL_2(\widehat{\Z})$. Thus there exists a positive integer $N$ such that $\im \rho_E = \pi^{-1}(\im \rho_{E,N})$. The smallest such $N$ is called the level of the adelic Galois representation. Similarly, the smallest positive integer $n$ such that $\im \rho_{E,m^{\infty}}= \pi^{-1}(\im \rho_{E,n})$ is called the level of the $m$-adic Galois representation associated to $E$. In fact, for any fixed integer $m$, there exists a bound on the level of $\rho_{E,m^{\infty}}$ that depends only on the degree of $k$. See \cite[Thm. 1.1]{CT13}, \cite[Thm 2.3]{CP18} in the case where $m$ is prime and \cite[Prop. 6.1]{BELOV} for the general case.

A consequence of Serre's Open Image Theorem is that, given a non-CM elliptic curve $E/k$, the mod $p$ Galois representation is surjective for all sufficiently large primes. In \cite{serre72}, Serre asked whether there might exist some uniform constant $C$ depending only on $k$ such that $\im \rho_{E,p}=\GL_2(\Z/p\Z)$ for all primes $p>C$ and \emph{all} non-CM elliptic curves $E/k$. In the case where $k=\Q$, both significant theoretical results and computational evidence make it appear likely that the answer is yes, and this is now often referred to as Serre's Uniformity Conjecture. It is even believed that $C$ can be taken to be $37$ in the case of non-CM elliptic curves over $\Q$. See, for example, \cite[Conj. 1.12]{ZywinaImages} and \cite[Conj. 1.1]{sutherland}.

If $\im \rho_{E,p}$ is not all of $\GL_2(\Z/p\Z)$, then it is contained in one of its known maximal subgroups. These include the Borel subgroup, the normalizer of a split or non-split Cartan subgroup, or an exceptional subgroup; see \cite[Section 2]{serre72} for details. For primes $p \leq 13$, the groups that arise as $\im \rho_{E,p}$ for a non-CM elliptic curve $E/\Q$ are known. The case of primes $p \leq 11$ was completed by Zywina \cite{ZywinaImages}; see also Sutherland \cite{sutherland}. At the time, the classification for $p=13$ was complete aside from ruling out the existence of non-CM elliptic curves $E/\Q$ with $\im \rho_{E,p}$ contained in the normalizer of a (split or non-split) Cartan subgroup. Baran \cite{Baran14} showed that such an elliptic curve would correspond to a rational point on an explicit genus 3 curve, and work of Balakrishnan, Dogra, M\"{u}ller, Tuitman, and Vonk \cite{Balakrishnan} showed that this genus 3 curve had no non-cuspidal, non-CM points. For a list of the groups that arise as $\im \rho_{E,p}$ for primes $p \leq 13$, as well as degrees of fields of definition for points of order $p$, see Tables 1 and 2 in \cite{GJN}. Throughout, we use the notation of Sutherland \cite{sutherland} to denote subgroups of $\GL_2(\Z/p\Z)$, which is also the notation used in LMFDB.

More generally, one could seek to classify which groups arise as $\im \rho_{E,p^{\infty}}$ for a non-CM elliptic curve $E/\Q$.\footnote{The case of a CM elliptic curve $E$ defined over $\Q(j(E))$ is addressed in recent work of Lozano-Robledo \cite{AlvaroCM}.} One of the first results in this direction was work of Rouse and Zureick-Brown \cite{RouseDZB} which gave the complete classification for $p=2$. The groups which arise infinitely often as $\im \rho_{E,p^{\infty}}$ were classified by Sutherland and Zywina \cite{SutherlandZywina}. For $p=3$, evidence suggests that the groups identified in \cite{SutherlandZywina} are in fact the only groups which arise; see forthcoming work of Rouse, Sutherland, and Zureick-Brown \cite{3adicimages}.

\subsection{Isogenies of elliptic curves}
Let $E/k$ be an elliptic curve, and let $P \in E(\overline{k})$ be a point of order $N$. If the subgroup generated by $P$ is fixed (as a group) by $\Gal_k$, then we say $E$ possesses a rational cyclic subgroup of order $N$. Alternatively, since such a subgroup is the kernel of a $k$-rational isogeny from $E$ to another elliptic curve defined over $k$, we may say $E$ has a $k$-rational cyclic $N$-isogeny. In the case of elliptic curves $E/\Q$, we have a complete determination of the rational cyclic subgroups that can occur.

\begin{thm}[Mazur \cite{mazur}, Kenku \cite{kenku}, and others; see Section 9 of \cite{LR}] \label{IsogClassification}
If $E/\Q$ is an elliptic curve possessing a $\Q$-rational cyclic subgroup of order $N$, then $N \leq 19$ or $N \in \{21, 25, 27, 37, 43, 67, 163\}$.
\end{thm}
Let $p \geq 5$ be prime, and let $E/\Q$ be a non-CM elliptic curve with a rational cyclic $p$-isogeny. Work of Greenberg, Rubin, Silverberg, and Stoll  \cite{greenberg2012}, \cite{GRSS14} shows that $\im \rho_{E,p^{\infty}}$ is as large as possible given the isogenies over $\Q$ with degree a power of $p$. In particular, their work implies the following theorem which plays a crucial role in the proof of our main results.

        \begin{thm}[Greenberg \cite{greenberg2012}, Greenberg, Rubin, Silverberg, Stoll \cite{GRSS14}] \label{isogenyTHM}
        Let $E/\Q$ be a non-CM elliptic curve with a $\Q$-rational cyclic isogeny of prime degree $p$.
        \begin{enumerate}
        \item If $p=7, 11,$ or $13$, then for any choice of basis the image of $\rho_{E,p^{\infty}}$ contains $I_2+pM_2(\Z_p)$. 
        \item Suppose $p=5$. If $E/\Q$ does not have a rational cyclic 25-isogeny, then for any choice of basis the image of $\rho_{E,5^{\infty}}$ contains $I_2+5M_2(\Z_5)$. Otherwise, the image of $\rho_{E,5^{\infty}}$ contains $I_2+25M_2(\Z_5)$.
        \end{enumerate}
                \end{thm}
                \begin{proof}
                Let $E/\Q$ be a non-CM elliptic curve with a rational cyclic $p$-isogeny. Note that any Sylow pro-$p$ subgroup of $\GL_2(\Z_p)$ contains $I_2+pM_2(\Z_p)$. Thus if $p=11$ or 13, the theorem statement can be deduced from Theorem 1 in \cite{greenberg2012} and the discussion which follows \cite[p.1186-1187]{greenberg2012}. For $p=7$, this is given by Theorem 5.5 in \cite{GRSS14}. So suppose $p=5$. If none of the elliptic curves in the $\Q$-isogeny class of $E$ has 2 independent isogenies of degree 5, then the statement follows from Theorem 2 of \cite{greenberg2012}. So suppose there exists an elliptic curve $\Q$-isogenous to $E$ with 2 independent isogenies of degree 5. Then either $E$ has a rational cyclic 25-isogeny, and the claim follows from Proposition 5.1.1 of \cite{greenberg2012}, or else $E$ has 2 independent isogenies of degree 5. Suppose the latter holds. Then $\rho_{E,5^{\infty}} : \Gal_{\Q} \to \GL_{2}(\Z_{5})$ has level $5^{r}$ which we identify with a subgroup $G$ of $\GL_{2}(\Z/5^{r} \Z)$. Let $K \subseteq G$ be the kernel of reduction map modulo $5$. Then, $[G : K]$ has order coprime to $5$ because $E$ has two independent $5$-isogenies (see Table 1 in \cite{GJN}). It follows that $K$ is a Sylow $5$-subgroup of $G$ and Theorem 2 of \cite{greenberg2012} gives that the index of $K$ in $\GL_{2}(\Z/5^{r} \Z)$ is divisible by $5$ but not $25$. If we let $L = \{ g \in \GL_{2}(\Z/5^{r} \Z) : g \equiv I \pmod{5} \}$, then $K \subseteq L$ and $|K| = |L|$. Thus, $K = L$ and the image of $\rho_{E,5^{\infty}}$ contains all matrices congruent to the identity modulo $5$.
                \end{proof}

\begin{rmk}
This shows that if $E/\Q$ is a non-CM elliptic curve with a rational cyclic $p$-isogeny for some prime $p \geq 5$, then $\im \rho_{E,p^{\infty}}$ is the complete pre-image of $\im \rho_{E,p^m}$ in $\GL_2(\Z_p)$, where $m$ is the maximum integer such that $E$ possesses a $\Q$-rational cyclic $p^m$-isogeny. This does not hold if $p=3$. For example, by Sutherland and Zywina \cite{SutherlandZywina} there exist non-CM elliptic curves $E/\Q$ such that the associated 3-adic Galois representation has level 27. However, no non-CM elliptic curves over $\Q$ have a rational cyclic 27-isogeny (see, for example, Table 4 in \cite{LR}). \end{rmk}

\subsection{Modular curves.} Here we briefly recall the constructions of the modular curves $X_1(N)$ and $X_0(N)$, along with some useful formulas regarding maps between modular curves. For more details, see, for example, \cite[$\S7.7$]{modular}, \cite{DiamondIm}, \cite[$\S6.7$]{shimura}, \cite{DR}.

For any $N \in \Z^+$, the curve $Y_1(N)$ parametrizes $\mathbb{C}$-isomorphism classes of elliptic curves with a distinguished point of order $N$. An explicit construction is given by
\[
Y_1(N) \coloneqq \mathbb{H}/\Gamma_1(N),
\]
where $\mathbb{H}$ denotes the upper half plane and 
\[
\Gamma_1(N) \coloneqq \left\{\begin{bmatrix} a & b \\ c & d \end{bmatrix} \in \text{SL}_2(\Z): c \equiv 0 \pmod{N} \text{ and } a \equiv d \equiv 1 \pmod{N} \right\}
\] acts on $\mathbb{H}$ via linear fractional transformations. The resulting Riemann surface is not compact. By adding in a finite number of points---the cusps---we obtain its compactification $X_1(N)$. This can be identified with a smooth projective curve defined over $\Q$.

            \begin{prop}\label{prop:Degree}
                For positive integers $a$ and $b$, there is a natural $\Q$-rational map $f:X_1(ab) \rightarrow X_1(a)$ defined by sending $[E,P]$ to $[E,bP]$. Moreover
                \[
                    \deg(f)=
                    c_{f}\cdot b^2 \prod_{p \mid b,\, p \nmid a}
                    \left(1-\frac{1}{p^2}\right),
                \]
                where $c_{f}=1/2$ if $a \leq 2$ and $ab>2$ and $c_{f}=1$ otherwise. 
            \end{prop}
\begin{proof}
The fact that the map is $\Q$-rational follows from the moduli interpretation. The degree calculation can be deduced from \cite[p.66]{modular}.
\end{proof}

Similarly, the curve $Y_0(N)$ parametrizes $\mathbb{C}$-isomorphism classes of elliptic curves with a cyclic subgroup of order $N$. Specifically, 
\[
Y_0(N) \coloneqq \mathbb{H}/\Gamma_0(N),
\]
where
\[
\Gamma_0(N) \coloneqq \left\{\begin{bmatrix} a & b \\ c & d \end{bmatrix} \in \text{SL}_2(\Z): c \equiv 0 \pmod{N} \right\}.
\] The compactification $X_0(N)$ can be identified with an smooth projective curve over $\Q$.

\begin{prop}
For a positive integer $N$, there is a natural $\Q$-rational map $f:X_1(N) \rightarrow X_0(N)$ defined by sending $[E,P]$ to $[E,\langle P \rangle]$. If $N\leq 2$, then $\deg(f)=1$. Otherwise $\deg(f)=\varphi(N)/2$, where $\varphi$ denotes the Euler phi function.
\end{prop}

\begin{proof}
As with the previous proposition, the fact that the map is $\Q$-rational follows from the moduli interpretation, and the degree calculation can be deduced from \cite[p.66]{modular}.
\end{proof}

\subsection{Closed points on curves.} \label{ClosedPointsBackground} Let $C$ be a curve defined over a number field $k$. Throughout, we consider closed points $x\in C$, which are $\Gal_k$-orbits of points in $C(\overline{k})$. By the degree of $x$ we mean the degree of the associated residue field $k(x)$ over $k$. Alternatively, the degree of $x$ is equal to the cardinality of the $\Gal_k$-orbit of points in $C(\overline{k})$ associated to $x$. In the case where $C=X_1(N)$, the following lemma gives a way to compute the degree of a closed point $x\in X_1(N)$ associated to a non-CM elliptic curve $E$ and point $P$ of order $N$. Here, $k(P)$ denotes the field extension of $k$ generated by the coordinates of $P$.

   \begin{lem}\label{save2}\cite[Lemma 2.1]{BELOV}
            Let $E$ be a non-CM elliptic curve defined over the number field $k=\Q(j(E))$, let $P\in E$ be a point of order $N$, and let $x = [E,P]\in X_1(N)$. Then
            \[
                \deg(x)=c_x[k(P):\Q],
            \]
            where $c_x=1/2$ if {$2P \neq O$ and} there exists $\sigma \in \Gal_k$ such that {$\sigma(P)=-P$} and $c_x=1$ otherwise.
        \end{lem}
        
    More generally, it is often useful to construct the residue field of a closed point on $X_1(N)$ using Weber functions. For an elliptic curve $E$, let $\mathfrak{h}: E \rightarrow E/\Aut(E) \cong \mathbb{P}^1$  be a Weber function. If $E: y^2=x^3+Ax+B$ and $P=(x,y) \in E$, then $\mathfrak{h}$ can be taken to be
\[ \mathfrak{h}(P) = \begin{cases} x & AB \neq 0 \\ x^2 & B = 0 \\ x^3 & A = 0 \end{cases}. \]
We have $B = 0$ iff $j(E) = 1728$ and $A = 0$ iff $j(E) = 0$. Then for $x = [E,P]\in X_1(N)$, the residue field $\Q(x)$ is
        \[
       \Q(j(E), \mathfrak{h}(P)). 
        \]It follows from \cite[Proposition VI.3.2]{DR} that there is a model of $E$ over $\Q(x)$ such that $P \in E(\Q(x))$.
        \subsection{Isolated points.}
Let $C/k$ be a curve, and suppose $P_0 \in C(k)$.\footnote{For the case where $C$ does not have a $k$-rational point, see $\S 4$ of \cite{BELOV}.} For any positive integer $d$, we let $C^{(d)}$ denote the $d$th symmetric product of $C$, a variety whose points correspond to effective divisors of degree $d$ on $C$. Any closed point $x \in C$ of degree $d$ gives a $k$-rational point of $C^{(d)}$, and we have a natural map to the Jacobian variety
\[
\Phi: C^{(d)} \rightarrow \text{Jac}(C)
\]
defined by sending $x=P_1 + \cdots +P_d$ to the divisor class $[P_1 + \cdots +P_d -dP_0]$, where $P_1, \dots, P_d$ denote the points in the $\Gal_k$ orbit $x$.

If $C$ has infinitely many closed points of degree $d$, then one of the following must be true:
\begin{enumerate}
\item $\Phi(x)=\Phi(y)$ for distinct closed points $x$ and $y$. As effective degree $d$ divisors, $x$ and $y$ have distinct support, so it follows there is a function $f$ of degree $d$ such that $\text{div}(f)=x-y$. Hence $f:C \rightarrow \mathbb{P}^1$ is a dominant morphism of degree $d$, and by Hilbert's irreducibility theorem \cite[Ch.9]{serre97} $f^{-1}(\mathbb{P}^1(k))$ contains infinitely many points of degree $d$. That is, there exists an infinite family of degree $d$ points ``parametrized by $\mathbb{P}^1$."
\item $\Phi$ is injective on the set of degree $d$ points. Since $\im \Phi$ is a closed subscheme of $\text{Jac}(C)$, Faltings's Theorem \cite{faltings} implies there exist a finite number of $k$-rational abelian subvarieties $A_i \subset \text{Jac}(C)$ and $k$-rational points $x_i \in \im \Phi$ such that
\[
(\im \Phi) (k) = \bigcup_{i=1}^n [x_i+A_i(k)].
\] Thus one of the $A_i$ has positive rank, and this gives an infinite family of degree $d$ points ``parametrized by $A_i$."
\end{enumerate}

Thus we see that the existence of infinitely many degree $d$ points implies we either have a degree $d$ function $f: C \rightarrow \mathbb{P}^1$ or else $\im \Phi$ contains the translate of a positive rank abelian subvariety of $\text{Jac}(C)$. In fact, the converse holds as well. Following \cite{BELOV}, we say a closed point $x\in C$ of degree $d$ is \textbf{isolated} if it does not belong to one of these infinite families of degree $d$ points, that is, if (1) there is no other point $y\in C^{(d)}(k)$ such that $\Phi(x) = \Phi(y)$ and (2) there is no positive rank abelian subvariety $A\subset \text{Jac}(C)$ such that $ \Phi(x) + A \subset \text{im}(\Phi)$. Moreover, we say points satisfying condition (1) are $\mathbb{P}^1$-isolated, and points satisfying condition (2) are AV-isolated. The following characterization of isolated points strengthens an observation of Frey \cite{frey}.

    \begin{thm}[Bourdon, Ejder, Liu, Odumodu, Viray, \cite{BELOV}]\label{thm:FiniteIsolated}
        Let $C$ be a curve over a number field.
        \begin{enumerate}
            \item There are infinitely many degree $d$ points on $C$ if and only if there is a degree $d$ point on $C$ that is \emph{not} isolated.
            \item There are only finitely many isolated points on $C$.
        \end{enumerate}
    \end{thm}
\noindent In particular, if there exist only finitely many points of degree $d$, then each degree $d$ point is isolated. However, having infinitely many degree $d$ points does not preclude the existence of additional isolated degree $d$ points. Some places in the literature use the term \textbf{sporadic} to denote a closed point $x \in C$ such that there are only finitely many points of degree at most $\deg(x)$. By Theorem \ref{thm:FiniteIsolated}, we see that every sporadic point is in fact an isolated point.

A key tool in studying isolated points is the following criterion for when the image of isolated points remain isolated. 
  \begin{thm}[Bourdon, Ejder, Liu, Odumodu, Viray, \cite{BELOV}]\label{LevelLowering}
    Let $f\colon C \to D$ be a finite map of curves and let $x\in C$ be an isolated point.  If $\deg(x) = \deg(f(x))\cdot\deg(f)$, then $f(x)$ is an isolated point of $D$.
  \end{thm}
  
  \subsection{CM elliptic curves} \label{CMsection}
Let $E$ be an elliptic curve defined over a number field $F$. We say $E$ has \textbf{complex multiplication}, or CM, if $\End_{\overline{F}}(E)$ is strictly larger than $\Z$. In this case, $\End_{\overline{F}}(E) \cong \OO$, an order in an imaginary quadratic field $K$. If $\OO_K$ denotes the ring of integers in $K$, then $\OO$ is a subring of $\OO_K$ of index $\ff$, where $\ff$ is called the \textbf{conductor} of $\OO$, and it is the unique subring of $\OO_K$ of this index. Explicitly, we have
    \[
  \OO=\Z+\ff\OO_K.
  \]
 Thus an order $\OO$ in $K$ can be uniquely determined by its \textbf{discriminant}
  \[
  \Delta \coloneqq \ff^2\Delta_K,
  \] where $\Delta_K$ denotes the discriminant of $K$.
See \cite[Lemma 7.2]{cox} for details.

If $E/F$ is an $\OO$-CM elliptic curve, then elements of $\Gal_{FK}$ commute with elements of $\OO$ in their action on $E[N]$, a free $\OO/N\OO$-module of rank 1 by \cite[Lemma 1]{Parish89}. This implies the mod $N$ Galois representation of $E/FK$ can be expressed as
\[
\rho_{E,N}:\Gal_{FK} \rightarrow \Aut_{\OO/N\OO}(E[N]) \cong (\OO/N\OO)^{\times}.
\]
Thus we may interpret the action of $\Gal_{FK}$ on $N$-torsion points of an $\OO$-CM elliptic curve as the action of $(\OO/N\OO)^{\times}$ on a free $\OO/N\OO$-module of rank 1. We denote $(\OO/N\OO)^{\times}$ by $C_N(\OO)$ and call it the \textbf{mod $N$ Cartan subgroup}. 

If we fix a model of $E$ defined over $K(j(E))$, the action of $\OO$ on points of $E$ is rationally defined. Denote by $\overline{E[N]}$ the orbits of points in $E[N]$ under the action of $\OO^{\times}$. The action of $\OO/N\OO$ on $E[N]$ induces an action of the \textbf{reduced mod N Cartan subgroup} $\overline{C_{N}(\OO)}$ on $\overline{E[N]}$, where
\[
\overline{C_{N}(\OO)}\coloneqq C_{N}(\OO)/q_N(\OO^{\times})
\]
and $q_N:\OO \rightarrow \OO/N\OO$ is the natural map. For any point $P \in E$ of order $N$, the degree of $K(j(E))(\mathfrak{h}(P))$ over $K(j(E))$ is equal to the size of the $\overline{C_{N}(\OO)}$-orbit of $\overline{P} \in \overline{E[N]}$. From this we can deduce the degree of $[E,P]$ on $X_1(N)$ viewed as a curve over $K$. See Section 7A of \cite{BC1} for details.
                
\section{Points of Odd Degree on $X_1(N)$}
In this section we will prove Theorem \ref{OddDegThm}. We begin with preliminary lemmas in $\S3.1\mhyphen 3.2$, and the theorem itself is proved in $\S3.3$. A refinement of Theorem \ref{OddDegThm} is given in $\S3.4$. A key observation is that aside from one exceptional $j$-invariant, to have a point $x=[E,P]\in X_1(n)$ of odd degree with $j(x) \in \Q$, there must exist a model of $E/\Q$ with a rational cyclic $p$-isogeny for all odd primes $p$ dividing $n$. Thus Theorem \ref{IsogClassification} significantly restricts the possibilities for $n$. In the case of CM elliptic curves, our results can be deduced from work of Aoki \cite{aoki95}. For non-CM elliptic curves, Theorem \ref{OddDegThm} follows from classification results for Galois representations of elliptic curves over $\Q$, as outlined in $\S2.1$, along with various computations which address special cases. In particular, many of the fiber product computations we require were originally performed by Daniels and Gonz\'{a}lez-Jim\'{e}nez \cite{DanielsGJ20}, \cite{DanielsGJ}. We also employ a useful result about lifting rational points due to Najman and Gonz\'{a}lez-Jim\'{e}nez \cite[Prop. 4.6]{GJN}.

\subsection{Connection with rational cyclic isogenies}   
        
\begin{lem} \label{OddDegLem}
Let $E/\Q$ be an elliptic curve and $P \in E(\overline{\Q})$ a point of order $p n$ where $p\geq 3$ is prime and $n \in \Z^+$. Then one of the following occurs:
\begin{enumerate}
\item $p \in \{3,5,7,11,13, 19, 43, 67, 163\}$ and $E$ has a rational $p$-isogeny,
\item $p=7$ and $j(E)= 3^3\cdot 5\cdot 7^5/2^7$, or 
\item the residue field of $[E,P] \in X_1(pn)$ has even degree.
\end{enumerate}
\end{lem}

\begin{proof} If $E$ has complex multiplication and the residue field of $[E,P] \in X_1(p n)$ has odd degree, then there is a number field $F$ of odd degree and a model of $E/F$ where $P \in E(F)$. By Aoki \cite[Cor. 9.4]{aoki95}, $E$ has CM by an order in $K=\Q(\sqrt{-p})$. Since $j(E) \in \Q$, the field $K$ has class number 1, and $p \in \{3,7,11,19,43,67,163\}$. Moreover the model of $E$ over $\Q$ has a rational cyclic $p$-isogeny; see for example \cite[Prop. 5.7]{BCS}. From now on we assume $E$ is non-CM and fix a model of $E/\Q$. 

If $\rho_{E,p}$ is surjective, then the residue field of $[E, nP] \in X_1(p)$ has even degree by \cite[Theorem 5.1]{LR} and Lemma \ref{save2}, and hence the residue field of $[E,P] \in X_1(p n)$ has even degree. Thus we may assume $\rho_{E,p}$  is not surjective. 

First, suppose $p \leq 13$. Then as discussed in $\S2.1$, the subgroups that arise as $\im \rho_{E,p}$ are known. By checking each case, see for example \cite[Tables 1 \& 2]{GJN} and Lemma \ref{save2}, we see that we are in case (i) or (iii) except when $p=7$ and $\im \rho_{E,7}$ is conjugate to 7Ns.2.1 or 7Ns.3.1 (here we use LMFDB label{s}, also following \cite{sutherland}). By \cite[Theorem 1.5]{ZywinaImages}, $\im \rho_{E,7}$ is conjugate to one of these groups only if $j(E)=3^3\cdot 5\cdot 7^5/2^7$.

Next, suppose $p \geq 17$. {Here, $\rho_{E,p}$ is known to have at least 4 possible images aside from $\GL_2(\Z/p\Z)$, each of which is contained in a Borel subgroup. For each of these,} we are in case (iii). (See, e.g., \cite[Table 2]{GJN} and Lemma \ref{save2}.) If $\im \rho_{E,p}$ is not {one of these known groups}, then it is conjugate to one of two subgroups {of the normalizer of a non-split Cartan subgroup. See for example  \cite[Theorem 3.2]{GJN}.} The degree of the residue field of $[E, nP]$ is even in either case by \cite[Theorem 5.6]{GJN} and Lemma \ref{save2}.
\qedhere 
\end{proof}

\subsection{Elliptic curves with {rational cyclic isogenies of degree 15 or 21}}

\begin{prop} \label{prop:ExtraJ}
Let $E/\Q$ be an elliptic curve, and suppose $[E,P] \in X_1(n)$ has odd degree.  
\begin {enumerate}
\item If $E$ has a rational cyclic 15-isogeny, then $n=2^a3^b$ or $2^a5^c$ where $a \leq 1$. 
\item If $E$ has a rational cyclic 21-isogeny, then $n=2^a3^b7^c$ where $a \leq 1$.
\end{enumerate}
\end{prop}

\begin{proof}
{If} $E$ has a rational cyclic 15-isogeny or 21-isogeny, then $j(E)$ is one of the eight values listed in \cite[Table 4]{LR} and we may pick a representative curve for each value. Each representative in case $(i)$ has a rational {cyclic $p$-isogeny for a prime $p$ iff $p=3$ or $5$. Similarly, each representative in case $(ii)$ has a rational cyclic $p$-isogeny iff $p=3$ or 7}. By Lemma \ref{OddDegLem}, the prime divisors of $n$ are $\{2,3,5\}$ or $\{2,3,7\}$, so we consider numbers of the form $n=2^a3^b5^c$ in case $(i)$ and $n=2^a3^b7^c$ in case $(ii)$. Computing the $15^{th}$ division polynomial for each representative in case $(i)$, we see that the point on $X_1(15)$ corresponding to $E$ has even degree. So $n=2^a3^b$ or $2^a5^c$. Finally, computing the $4^{th}$ division polynomial for all eight representatives shows that the point on $X_1(4)$ corresponding to $E$ has even degree. Thus $a \leq 1$ in both cases.
\end{proof}

{\begin{rmk}
Computing the $21^{st}$ division polynomial for each representative in case $(ii)$, we see that the point on $X_1(21)$ corresponding to $E$ can have odd degree. 
\end{rmk}}

\subsection{Proof of Theorem \ref{OddDegThm}}
Let $x=[E,P] \in X_1(n)$ be a point of odd degree with $j(x) \in \Q$. First suppose $E$ has complex multiplication. Then there is a number field $F$ of odd degree and a model of $E/F$ where $P \in E(F)$. By Aoki \cite[Cor. 9.4]{aoki95}, $n=2^ap^b$ for an odd prime $p$ and $a \leq 2$. If $b>0$, then $a \leq 1$ and $E$ has CM by an order in $K=\Q(\sqrt{-p})$. Since $j(E) \in \Q$, it follows that $K$ has class number 1, and so $p \in \{3,7,11,19, 43, 67, 163\}$. Moreover any model of $E$ over $\Q$ has a rational cyclic $p$-isogeny; see for example \cite[Prop. 5.7]{BCS}.

From now on we will assume $E$ is non-CM. We fix a model of $E/\Q$. First suppose $j(E) \neq 3^3\cdot 5\cdot 7^5/2^7$. If $p \mid n$ where $p \geq 3$ is prime, then $E$ has a rational cyclic $p$-isogeny by Lemma \ref{OddDegLem}. Since $E$ is non-CM, $p \in \{3,5,7,11,13\}$ (see, for example, \cite[Table 4]{LR}). If $p_1$ and $p_2$ divide $n$ where $p_i \geq 3$ are distinct primes, then $E$ has a rational cyclic $p_1p_2$-isogeny. By Theorem \ref{IsogClassification}, this cannot happen unless $E$ has a rational cyclic isogeny of degree 15 or degree 21. Such elliptic curves are addressed in Proposition \ref{prop:ExtraJ}. Thus $n=2^ap^b$ for $p \in \{3,5,7,11,13\}$ and $E$ has a rational cyclic $p$-isogeny.

We next address the exponent of 2. First suppose $E$ has a rational cyclic 2-isogeny, which means that any point of order 2 defined over an extension of odd degree is in fact defined over $\Q$. By \cite[Prop. 4.6]{GJN}, any point on $E$ of order $2^a$ corresponding to a point on $X_1(2^a)$ of odd degree must in fact be a point in $X_1(2^a)(\Q)$. By \cite{Levi}, there is no elliptic curve over $\Q$ with a rational point of order 16, so $a \leq 3$. If $E$ does not have a rational cyclic 2-isogeny, then by the classification of 2-adic images of non-CM elliptic curves over $\Q$ due to Rouse and Zureick-Brown \cite{RouseDZB}, a point of order 4 will occur in even degree unless $E$ corresponds to a rational point on the curve labeled X20 by \cite{RouseDZB}. See the data associated to \cite[Cor. 3.4 and Cor. 3.5]{GJLR17} showing the degrees of a point of order $2^a$ on an elliptic curve defined over $\Q$ based on the Rouse-Zureick-Brown classification.  We note in particular that $a \leq 2$ if $E$ has no 2-isogeny. 

It remains to show that if $b>0$, then $a\leq2$. By the previous paragraph, we may assume $E$ has a rational 2-isogeny. Since $E$ is non-CM elliptic curve which also has a rational cyclic $p$-isogeny, Theorem \ref{IsogClassification} and \cite[Table 4]{LR} imply $p \leq 5$. 
\begin{itemize}
\item Suppose $p=5$. Since $E$ gives a rational point on $X_1(2)$ and $\deg(X_1(4) \rightarrow X_1(2))=2$ by Proposition \ref{prop:Degree}, $E$ gives a point on $X_1(4)$ of degree 1 or 2. If it is in degree 1, then $E$ has a 4-isogeny, which contradicts Theorem \ref{IsogClassification}. Thus any point on $X_1(4)$ corresponding to $E$ has even degree and $a \leq 1$.
\item Suppose $p=3$. As in the previous case, the only way $E$ can give a point of odd degree on $X_1(4)$ is if it gives a rational point on $X_1(4)$. By \cite[Prop. 4.6]{GJN}, the only way $E$ can give a point on $X_1(8)$ of odd degree is if it is in degree 1, which cannot happen by Theorem \ref{IsogClassification}. Thus $a \leq 2$. 

\end{itemize}

If $j(E)= 3^3\cdot 5\cdot 7^5/2^7$, it suffices to pick a particular elliptic curve $E/\Q$ with this $j$-invariant. Since $E$ has no rational isogenies, $\Supp(n) \subseteq\{2, 7\}$ by Lemma \ref{OddDegLem}. Computing division polynomials shows that any point on $X_1(4)$ corresponding to $E$ has degree 6, so $a \leq 1$.

\subsection{Refined bounds on exponent of 2}
Often we may improve the bound on the exponent of 2 found in Theorem \ref{OddDegThm}.
\begin{prop}\label{OddDegProp}
Let $x\in X_1(2^ap^b)$ be a point of odd degree where $a,b$ are nonnegative integers and $p \geq 5$ is prime. Suppose 
\[
j(x) \in \Q \setminus \{-3^3\cdot 13 \cdot 479^3/2^{14}, 3^3\cdot 13/2^2\}.
\] If $b>0$, then $a \leq 1$. 

\end{prop}
\begin{proof}
Let $x=[E,P] \in X_1(2^ap^b)$ be as in the theorem statement, where $b >0$. If $E$ has CM, the claim follows from Aoki \cite[Cor. 9.4]{aoki95}, so henceforth we assume $E$ is non-CM and fix a model of $E/\Q$. If $j(E) = 3^3\cdot 5\cdot 7^5/2^7$, then computing division polynomials shows that any point on $X_1(4)$ corresponding to $E$ has degree 6, so $a \leq 1$. If $j(E) \neq 3^3\cdot 5\cdot 7^5/2^7$, then by Lemma \ref{OddDegLem}, $E$ has a rational $p$-isogeny, and so $p \in \{5,7,11,13\}$ by \cite[Table 4]{LR}. 

If $p=5$ and $E$ has a rational 2-isogeny, then $a\leq 1$ by the proof in $\S3.3$. If $p >5$, then $E$ has no rational 2-isogeny by Theorem \ref{IsogClassification} and \cite[Table 4]{LR}. So suppose $E$ has no rational 2-isogeny. By the classification of 2-adic images due to Rouse and Zureick-Brown \cite{RouseDZB}, any point on $X_1(4)$ corresponding to $E$ is of even degree unless $E$ corresponds to a rational point on X20. So it suffices to consider the fiber product of X20 and $X_0(p)$. We consider each prime separately:
\begin{enumerate}
\item If $p=5$, then Daniels and Gonz\'{a}lez-Jim\'{e}nez \cite[Proposition 6(k)]{DanielsGJ20} show the fiber product of X20 and $X_0(5)$ has only cusps. So $a \leq 1$. 
\item If $p=7$, Daniels and Gonz\'{a}lez-Jim\'{e}nez \cite[Proposition 6(s)]{DanielsGJ20} compute the fiber product of $X_0(7)$ and X20. They show the non-cuspidal rational points on this curve correspond to $j$-invariants $-3^3\cdot 13 \cdot 479^3/2^{14}$ and $3^3\cdot 13/2^2$, which appear in the theorem statement. So aside from these two $j$-invariants, $a \leq 1$.
\item If $p=11$, there are only a finite number of elliptic curves over $\Q$ with a rational cyclic 11-isogeny. See \cite[Table 4]{LR}. Computing division polynomials shows that $a \leq 1$.
\item If $p=13$, it will suffice to show the fiber product of $X_0(13)$ and X3 has no non-cuspidal rational points, since X20 covers X3. By Daniels and Gonz\'{a}lez-Jim\'{e}nez \cite[Table 8]{DanielsGJ}, this curve only has 2 rational points, and both are cusps. Indeed, there are two rational cusps $0, \infty$ on $X_0(13)$, and X3 $\cong \mathbb{}P^1_{\mathbb{\Q}}$. This means there are two cuspidal points $(0, \infty)$, $(\infty, \infty)$ in the fiber product and $a \leq 1$.\qedhere
\end{enumerate}
\end{proof}

\section{Non-CM Isolated Points in Odd Degree}
Here, we build on the results of Section 3 to prove that the non-CM $j$-invariants in $j(\mathcal{I}_{odd}) \cap \Q$ are $-3^2\cdot 5^6/2^3$ and $3^3\cdot 13/2^2$, giving the non-CM part of Theorem \ref{MainThm}. In $\S4.1$, we address the case of elliptic curves $E/\Q$ with a rational cyclic 25-isogeny. The argument uses constraints on $\im \rho_{E,5^{\infty}}$ due to Greenberg \cite{greenberg2012} in addition to work of Jeon, Kim, and Schweizer concerning intermediate modular curves \cite{JeonKim}, \cite{JKS20}. Following this, in $\S4.2\mhyphen4.5$, we show that aside from the two exceptional $j$-invariants noted above, any isolated point $x \in X_1(n)$ corresponding to a non-CM elliptic curve with $\deg(x)$ odd and $j(x) \in \Q$ must {map to an isolated point} on $X_1(54)$ or $X_1(162)$. We use explicit bounds on the level of the $m$-adic Galois representation as computed in \cite{BELOV}, {which at times can be improved to account for known instances of entanglement and constraints on ramification in torsion point fields (see Lemma \ref{extraLem1}),} in addition to classification results concerning the 2-adic \cite{RouseDZB} and 3-adic \cite{3adicimages} images of Galois representations of non-CM elliptic curves. In $\S4.6$, we show that any isolated point $x$ on $X_1(54)$ or $X_1(162)$ with $\deg(x)$ odd and $j(x) \in \Q$ would correspond to a rational point on an explicit genus 4 curve, but in fact all such rational points correspond to cusps. This leaves only points associated to $j$-invariants $-3^2\cdot 5^6/2^3$ and $3^3\cdot 13/2^2$. The first corresponds to an isolated point of degree 3 on $X_1(21)$ identified by Najman \cite{najman16}. In $\S4.7$, we show that there is a point $x \in X_1(28)$ of degree 9 with $j(x)= 3^3\cdot 13/2^2$ such that the associated Riemann-Roch space is 1-dimensional. Since the Jacobian of $X_1(28)$ has rank 0, this is enough to conclude the point is isolated.

\subsection{Elliptic curves with a rational cyclic 25-isogeny}
\begin{prop} \label{25lem}
Let $x \in X_1(2^a5^b)$ be a point corresponding to a non-CM elliptic curve with $\deg(x)$ odd and $j(x) \in \Q$. If there exists $y \in X_0(25)(\Q)$ with $j(y)=j(x)$, then $x$ is not isolated.
\end{prop}
\begin{proof} Suppose by way of contradiction that $x$ is isolated, and fix a model for $E/\Q$. If $b=0$, then $a \leq 3$ by Theorem \ref{OddDegThm} and $X_1(2^a)$ has genus 0 (and thus $x$ is not isolated). So we may assume $b>0$. Then $a \leq 1$ by Proposition \ref{OddDegProp}. Suppose first that $a = 0$. Since $X_1(5)$ has genus 0 (and hence has no isolated points), we may assume $b>1$. Let $f: X_1(5^b) \rightarrow X_1(25)$ be the natural map. By Theorem \ref{isogenyTHM},  $\im\rho_{E,5^\infty} =  \pi^{-1}(\im\rho_{E,5^2})$ and so $\deg(x) = \deg(f)\deg(f(x))$. By Theorem \ref{LevelLowering}, since $x$ is isolated, $f(x) \in X_1(25)$ is also isolated. We claim that $\deg(f(x)) = 5^r$ for some $r \in \Z^{+}$.
\par Let $f(x) = [E,P] \in X_1(25)$. Since $f(x)$ has odd degree, the classification of images of mod 5 Galois representations  {(see, for example, Tables 1 in \cite{GJN})} shows it corresponds to a point of degree 1 or 5 on $X_1(5)$. Let $y \in X_1(5)$ be the point corresponding to $f(x)$ and consider the tower of fields $\Q \subseteq \Q(y) \subseteq \Q(5P) \subseteq \Q(P)$. By Lemma \ref{save2}, $[\Q(5P): \Q(y)] \le 2$. By {Proposition 4.6 in \cite{GJN}},  $[\Q(P):\Q(5P)]$ divides either $5^2$ or $4\cdot 5$. Thus $[\Q(P): \Q]$ divides $8\cdot 125$. As $\Q(f(x)) \subseteq \Q(P)$,  $\deg(f(x)) = 5^k$ for some $k \le 3$ (since by assumption $\deg(f(x))$ is odd).  By Mazur's result on torsion points over $\Q$ \cite{mazur77}, $\deg(f(x)) \neq 1$. If the degree is $5^2$ or $5^3$, then the dimension of the Riemann-Roch space $L(f(x))$ is at least 14 {since the genus of $X_1(25)$ is 12. It follows that $f(x)$ is not $\mathbb{P}^1$-isolated.} Thus we must have $\deg(f(x)) = 5$.
\par We will next show $[E, \langle P \rangle] \in X_0(25)(\Q)$. Suppose not. Then {since the residue field of this point is a subfield of $\Q(f(x))$}, the degree of $[E, \langle P \rangle]$ must be 5. By assumption, $E$ corresponds to a rational point  on $X_0(25)$, so there must be a point $Q \in E$ of order 25 such that $[E, \langle Q \rangle] \in X_0(25){(\Q)}$. As $Q$ and $P$ both have order $25$, the group $G \coloneqq \langle Q, P \rangle$ is isomorphic to one of $\Z/25\Z \times \Z/5^s\Z$, $s=0, 1$, or $2$. We consider cases according to $\langle Q \rangle \cap \langle P \rangle$. 
If $\langle Q \rangle$ and $\langle P \rangle$ have nontrivial intersection, then either $\langle Q \rangle \cap \langle P \rangle = \langle Q \rangle$ in which case, $[E, \langle P \rangle]$ is $\Q$-rational, contradicting our assumption, or $\vert \langle Q \rangle \cap \langle P \rangle\vert = 5$. As $\langle Q \rangle$ and $ \langle P \rangle$ are each cyclic of order 25, they contain a unique subgroup of order 5, and thus $\langle 5Q \rangle =\langle Q \rangle \cap \langle P \rangle =  \langle 5P \rangle$. {Since $[E, \langle 5Q \rangle] \in X_0(5)(\Q)$, the group $\langle 5P \rangle$ is $\Q$-rational.} Let $\phi: X_0(25) \rightarrow X_0(5)$ be the natural map (note that $\phi$ has degree 5). Then $[E, \langle P \rangle]$ and $[E, \langle Q \rangle]$ are in the support of $\phi^*([E, \langle 5P \rangle])$, which means $\deg(\phi^*([E, \langle 5P \rangle]) \ge 1 + 5$. Since $\deg(\phi^*(y)) = \deg(\phi)\deg(y) = 5\cdot 1$ for any closed point $y \in {X_0(5)}(\Q)$, we have reached a contradiction. 
\par If $\langle Q \rangle$ and $\langle P \rangle$ have trivial intersection, then $G$ is isomorphic to $\Z/25\Z \times \Z/25\Z$. Since {the isogeny character associated to a cyclic subgroup of order $N$ can be trivialized over an extension of degree dividing $\varphi(N)$, we have $[\Q(Q): \Q] \vert \varphi(25)$. Moreover, $[\Q(P):\Q] \vert 10$ since by assumption, $[E, P]$ has degree 5 and $P$ requires at most a degree 2 extension of this degree 5 extension. Thus $F = \Q(Q,P) = \Q(E[25])$ has degree at most 200.} Since $5^4 \nmid [F:\Q]$,
\\[.5\baselineskip] \centerline{$\displaystyle \left[\mbox{GL}_2(\Z/25\Z):\im \rho_{E, 25}\right] = \frac{5^5\cdot3\cdot2^5}{[F:\Q]}$} $\;$
\\ 
is divisible by $5^2$. {This contradicts} Theorem 2 in \cite{greenberg2012}. Thus we may assume $[E, \langle P \rangle] \in X_0(25)(\Q)$.
\par Consider the map 
\\[.5\baselineskip] \centerline{$X_1(25) \rightarrow  X_{\Delta_2}(25) \rightarrow X_0(25)$,} $\;$
\\ where $X_{\Delta_2}(25)$ denotes the intermediate modular curve associated to $\Delta_2 = \{\pm1, \pm 4, \pm 6, \pm 9, \pm11\}$. {See \cite{JeonKim}, \cite{JKS20} for more on intermediate modular curves including the degrees of natural maps and genus information.} Since $\deg(X_{\Delta_2}(25) \rightarrow X_0(25)) = 2$, $[E, {P} ]$ must correspond to a degree 1 point under the map $X_1(25) \rightarrow X_{\Delta_2}(25)$  (else $\deg(f(x))$ and thus $\deg(x)$ is even).  Since this map $X_1(25) \rightarrow X_{\Delta_2}(25)$ is of degree 5, {Theorem \ref{LevelLowering}} implies the image of $f(x)$ is isolated on $X_{\Delta_2}(25)$; as this curve has genus 0, this is impossible.
\par  {Next, suppose $a=1$. So $x=[E,P] \in X_1(2\cdot 5^b)$ is isolated. Since $X_1(10)$ has genus 0, we may assume $b>1$.} Let $g: X_1(2\cdot 5^b) \rightarrow X_1(50)$ denote the natural map. We will first show that $\deg(x) = \deg(g) \deg(g(x))$. By {Theorem \ref{isogenyTHM}}, $x$ maps to a point $y = [E, 2P] \in X_1(5^b)$ such that $\deg(y) = \deg(f)\deg(f(y))$ where $f$ is the natural map $f: X_1(5^b) \rightarrow X_1(25)$. By Proposition \ref{prop:Degree}, we have that $5^{2b-4}$ divides $[\Q(x):\Q(f(y))] = [\Q(x): \Q(h(g(x)))]$, where $h: X_1(50)\rightarrow X_1(25)$. Since $\deg(h)= 3$ (again, by Proposition \ref{prop:Degree}), it follows that $5^{2b-4}$ divides $[\Q(x):\Q(g(x))]$, and as $[\Q(x): \Q(g(x))] \le \deg(g) = 5^{2b-4}$, it follows that $[\Q(x): \Q(g(x))] = \deg(g)$ and $\deg(x) = \deg(g)\deg(g(x))$. Thus $g(x) \in X_1(50)$ is isolated by Theorem \ref{LevelLowering}.  Next, by the assumption that $\deg(x)$ {is odd} we have that $\deg(g(x))$ is odd.  Then, since $[\Q(g(x)): \Q(h(g(x)))] \le \deg (h) = 3$, either $\deg(g(x)) = \deg(h(g(x))$ or $\deg(g(x)) = 3\cdot \deg(h(g(x)) = \deg(h)\deg(h(g(x)))$. We will show that $\deg(g(x)) = \deg(h)\deg(h(g(x)))$. Suppose by way of contradiction that $\deg(g(x)) = \deg(h(g(x)))$. By the argument given above, $\deg(h(g(x))) = 5^k$ for some $k \in \Z^+$.  This implies $E$ corresponds to a point on $X_1(2)$ of degree dividing $5^k$. Since $\deg(X_1(2) \rightarrow X_1(1))=3$, it follows that $E$ has a 2-isogeny over $\Q$. By assumption $E$ has a $25$-isogeny over $\Q$, {so this implies} $E$ has a $50$-isogeny over $\Q$, {contradicting Theorem \ref{IsogClassification}}. Thus $\deg(g(x)) = \deg(h)\deg(h(g(x)))$. By Theorem \ref{LevelLowering}, $h(g(x))$ is an isolated point on $X_1(25)$, but as shown above,  there are no odd degree isolated points on $X_1(5^b)$ for any $b \in \Z^+$.
\end{proof}

\subsection{Elliptic curves with a rational cyclic 21-isogeny} In Proposition \ref{Prop:3j}, we show that there are no isolated points of odd degree on $X_1(n)$ corresponding to elliptic curves with $j$-invariant $3^3\cdot5^3/2, \,-3^2\cdot 5^3 \cdot 101^3/2^{21},$ or $-3^3\cdot 5^3 \cdot 383^3/2^7$. This relies on the following lemma, where we give improved bounds on the level of $\rho_{E,14 \cdot 3^{\infty}}$ using the approach of Prop. 6.1 in \cite{BELOV}. 

\begin{lem} \label{extraLem1}
Let $E/\Q$ be an elliptic curve with LMFDB label 162.c1, 162.c2, or 162.c4. Then 
$
\im \rho_{E,14 \cdot 3^{\infty}} =\pi^{-1}(\im \rho_{E,14 \cdot 3^2})$ and
$\im \rho_{E,7 \cdot 3^{\infty}} =\pi^{-1}(\im \rho_{E,7 \cdot 3^2})$.

\end{lem}

\begin{proof}
Let $E/\Q$ be one of the curves listed above. Magma confirms $\Q(\zeta_9)^+$ is a subfield of one of the points on $X_1(7)$ associated to $E$, so
$
\Q(\zeta_9)^+ \subseteq \Q(E[7]) \cap \Q(E[9])
$
{by the Weil pairing}.
 Following the proof of \cite{BELOV}, Prop. 6.1, for all $s \in \Z^+$ we let 
 \begin{align*}
 L_s & \coloneqq \ker(\im \rho_{E,14\cdot 3^s} \rightarrow \im \rho_{E,3^s}),\\
 K_s & \coloneqq \ker(\im \rho_{E,14 \cdot3^s} \rightarrow \im \rho_{E,14}),\\
 K & \coloneqq \ker(\im \rho_{E,14 \cdot 3^{\infty}} \rightarrow \im \rho_{E,14}).
 \end{align*}
We may view $L_s$ as a subgroup of $\im \rho_{E,14}$ and $K_s$ as a subgroup of $\im \rho_{E,3^s}$. 
Moreover, we have the following diagram, where the vertical isomorphisms follow from Goursat's Lemma.
        \begin{equation*}\label{eq:diag2}
        \xymatrix{
            \im \rho_{E, 3^{s}}/K_{s} \ar@{->>}[r] \ar[d]^{\cong} &  \im \rho_{E, 3}/K_{1} \ar[d]^{\cong} \\
            \im \rho_{E, 14} / L_{s} \ar@{->>}[r]
             & \im \rho_{E, 14} / L_{1}}
        \end{equation*}
        The kernel of the top map is a power of 3, and so the kernel of the bottom map is as well. Thus $[L_1:L_{s}]$ is a power of 3, and more generally $[L_{s_1}:L_{s_2}]$ is a power of 3 for all $ 1 \leq s_1 \leq s_2$.

        We will show the maximal chain of proper containments $L_1 \supsetneq L_2 \supsetneq \dots \supsetneq L_r$ has length $r=2$. Magma confirms {$[\Q(E[2]):\Q]=6$}, and moreover that {$\Q(E[2]) \cap \Q(E[9])=\Q$}. As above, we let
 \begin{align*}
 L_s' & \coloneqq \ker(\im \rho_{E,2\cdot 3^s} \rightarrow \im \rho_{E,3^s}),\\
 K_s' & \coloneqq \ker(\im \rho_{E,2 \cdot3^s} \rightarrow \im \rho_{E,2}),\\
 K' & \coloneqq \ker(\im \rho_{E, 2 \cdot 3^{\infty}} \rightarrow \im \rho_{E,2}).
 \end{align*}
As before, $[L_{s_1}':L_{s_2}']$ is a power of 3 for all $1 \leq s_1 \leq s_2$. {Since $\Q(E[2]) \cap \Q(E[9])=\Q$, we have $\#L_2'=\#L_1'=6$. It follows that $L_1'=L_2'$, which gives the following diagram.} 
        \begin{equation*}
        \xymatrix{
            \im \rho_{E, 3^{2}}/K_{2}' \ar@{->>}[r] \ar[d]^{\cong} &  \im \rho_{E, 3}/K_{1}' \ar[d]^{\cong} \\
            \im \rho_{E, 2} / L_{2}' \ar@{=}[r]
             & \im \rho_{E, 2} / L_{1}'}
        \end{equation*}
        
{Since Magma shows $[\Q(E[9]):\Q(E[3])]=3^4$, it follows that $ \im \rho_{E, 3^{2}}=\pi^{-1}( \im \rho_{E, 3})$. Thus} the diagram implies 
\[
\ker(K' \text{ mod }{3^2} \rightarrow K' \text{ mod }{3})=I+\text{M}_2(3\Z/3^2\Z).
\]
By Prop. 3.5 of \cite{BELOV},
\[
\ker(K' \rightarrow K' \text{ mod }{3})=I+3 \text{M}_2(\Z_3),
\]
and so $\im \rho_{E,2\cdot 3^{\infty}} = \pi^{-1}(\im \rho_{E,2\cdot3})$. This means {$\Q(E[2]) \cap\Q(E[3^s])=\Q$ for all $s$.} 

{Let $(\alpha_i,0)$ for $1 \leq i \leq 3$ be the points on $E$ of order 2. Magma confirms} 
$
\Q(\zeta_7)^+\Q(\alpha_i)
$
has degree 9 over $\Q$, and the only subfields of degree 3 are $\Q(\zeta_7)^+$ and $\Q(\alpha_i)$. Neither $\Q(\alpha_i)$ nor $\Q(\zeta_7)^+$ is a subfield of $\Q(E[3^s])$ since $E$ has good reduction at 7, so their compositum is linearly disjoint from $\Q(E[3^s])$. Thus $3^2$ divides the size of $L_s$ for all $s$. {Since $\ord_3(\#L_1) \leq \ord_3(\# \GL_2(\Z/14\Z))=3$ and $[L_{s_1}:L_{s_2}]$ is a power of 3 for all $ 1 \leq s_1 \leq s_2$,} we have $r \leq 2$.

Thus the maximal chain of proper containments $L_1 \supsetneq L_2 \supsetneq \dots \supsetneq L_r$ has length $r=2$. In particular, we have the following diagram.
        \begin{equation*}\label{eq:diag2}
        \xymatrix{
            \im \rho_{E, 3^{3}}/K_3 \ar@{->>}[r] \ar[d]^{\cong} &  \im \rho_{E, 3^2}/K_2\ar[d]^{\cong} \\
            \im \rho_{E, 14} / L_3 \ar@{=}[r]
             & \im \rho_{E, 14} / L_2}
        \end{equation*}
    As above, this implies    
    \[
\ker(K \text{ mod }{3^3} \rightarrow K \text{ mod }{3^2})=I+\text{M}_2(3^2\Z/3^3\Z).
\]
Prop. 3.5 of \cite{BELOV} now gives that
\[
\ker(K \rightarrow K \text{ mod }{3^2})=I+3^2 \text{M}_2(\Z_3),
\]
and so
$
\im \rho_{E,14 \cdot 3^{\infty}} =\pi^{-1}(\im \rho_{E,14 \cdot 3^2})$. The second claim follows immediately.
\end{proof}

\begin{prop} \label{Prop:3j}
{If $x \in X_1(2^a3^b7^c)$ is a point of odd degree with $b,c>0$} and $j(x) \in \{3^3\cdot5^3/2,-3^2\cdot 5^3 \cdot 101^3/2^{21},-3^3\cdot 5^3 \cdot 383^3/2^7\}$, then $x$ is not isolated.
\end{prop}

\begin{proof}
Let {$x=[E,P] \in X_1(2^a3^b7^c)$} be {such} an isolated point, and choose the model of $E/\Q$ labeled 162.c1, 162.c2, or 162.c4. Note $E$ has a rational $3$-isogeny and a rational $7$-isogeny. From the classification in \cite{3adicimages}, we see that $\rho_{E,3^{\infty}}$ has level 3, and by Theorem \ref{isogenyTHM}, $\rho_{E,7^{\infty}}$ has level $7$.  Then by Proposition 6.1 in \cite{BELOV},
\[
\im \rho_{E, 42^{\infty}} = \pi^{-1}(\im \rho_{E, 2^{\alpha}3^{\beta}7}).
\]
where $\beta \leq 4$. Let $g: X_1(2^a3^b7^c) \rightarrow X_1(\gcd(2^a3^b7^c, 2^{\alpha}3^{\beta}7))$ be the natural map. By Proposition 5.8 in \cite{BELOV}, 
\[
\deg(x)=\deg(g)\deg(g(x)),
\]
and so $x$ maps to an isolated point on $X_1(\gcd(2^a3^b7^c, 2^{\alpha}3^{\beta}7))$ by Theorem \ref{LevelLowering}. We have $a \leq 1$ {by Proposition \ref{prop:ExtraJ}, and since we have assumed $b,c>0$, the possibilities for $\gcd(2^a3^b7^c, 2^{\alpha}3^{\beta}7)$ are} $3^d\cdot 7$ or $2\cdot 3^d \cdot 7$ for $1 \leq d \leq 4$. By Lemma  \ref{extraLem1} and Theorem \ref{LevelLowering}, we may further assume $d\leq 2$. We consider each $j$-invariant separately:
\begin{itemize}
\item Suppose $E=162.c1$. Factoring division polynomials shows that any odd degree point on $X_1(n')$ for $n'\in \{21, 63, 42, 126\}$ must have degree 63, 567, 189, or 1701, respectively. Since $567=9\cdot 63$ and $1701=9 \cdot 189$, by Theorem \ref{LevelLowering}, we need only consider $n'=3 \cdot 7$ and $n'= 2 \cdot 3 \cdot 7$. However, $X_1(21)$ has genus 5 and $X_1(42)$ has genus 25, so points of degree 63 and 189 (respectively) cannot be isolated {since their associated Riemann-Roch spaces have dimension at least 59}. We have reached a contradiction.
\item Suppose $E=162.c2$. Here, any point of odd degree on $X_1(n')$ must have degree 21, 189, 63, or 567, respectively. As in the previous case, we reach a contradiction.
\item Suppose $E=162.c4$. In this case, an odd degree point on $X_1(n')$ will have degree 9, 81, 27, or 243, respectively. Again, we will reach a contradiction. \qedhere
\end{itemize}
\end{proof}

\subsection{On the level the Galois representations for $E/\Q$ with $j(E)=3^3\cdot5\cdot7^5/2^7$}

\begin{lem} \label{OneJLevel}
If $E/\Q$ is an elliptic curve with $j(E)=3^3\cdot5\cdot7^5/2^7$, then for any choice of basis the image of $\rho_{E,7^{\infty}}$ contains $I_2+7M_2(\Z_7)$.
\end{lem}
\begin{proof}
  Let $E : y^{2} + xy = x^{3} - x^{2} - 107x - 379$ be a particular elliptic curve
  $E/\Q$ with $j(E) = 3^{3} \cdot 5 \cdot 7^{5}/2^{7}$. There is a degree $9$ extension of $\Q$ which contains the $x$-coordinate
  of a point of order $7$ on $E$. Theorem 1.5 of \cite{ZywinaImages} shows
  that the mod $7$ image of Galois for $E$ is the subgroup $H$ generated
  by $\begin{bmatrix} 2 & 0 \\ 0 & 4 \end{bmatrix}$,
  $\begin{bmatrix} 0 & 2 \\ 1 & 0 \end{bmatrix}$, and $\begin{bmatrix} -1 & 0 \\ 0 & -1 \end{bmatrix}$. By \cite[Part I, $\S6$, Lemmas 2 \& 3]{LT76}, it will suffice to show that $\im \rho_{E,49}$ is the complete preimage of $\im \rho_{E,7}$.
  
  If not, then
  the mod $49$ image is contained in a maximal subgroup of the mod $49$ preimage
  of $H$, which we denote $\tilde{H}$. This subgroup has $8$ maximal subgroups. We verify that for $\tilde{H}$, the set $\{ ({\rm trace} (g),\det(g)) : g \in \tilde{H} \}$ has $483$ elements, while for every maximal subgroup $M$, the
  size of $\{ ({\rm trace} (g), \det(g)) : g \in M \}$ has at most $357$ elements. For any prime $p \neq 7$ of good reduction,
\begin{align*}
  {\rm trace} (\rho_{E,49}({\rm Frob}_{p}) )&\equiv a_{p}(E) \pmod{49}\\
  \det (\rho_{E,49}({\rm Frob}_{p})) &\equiv p \pmod{49}.
\end{align*}
We compute $a_{p}(E)$ and compute the many pairs in the set $\{ ({\rm trace} (g), \det(g)) : g \in {\rm im}~\rho_{E,49} \}$. Testing all primes $p \leq 10^{5}$, we find that the latter set contains
  $483$ entries, and this proves that the image of $\rho_{E,49}$ must be
$\tilde{H}$. \href{http://users.wfu.edu/rouseja/isolated/odd7.txt}{See the website of the third author for the Magma code used.} \end{proof}

\subsection{Isolated points on $X_1(2^ap^b)$ for $p>3$}

\begin{thm} \label{IsolPtsThm}
Let $x \in X_1(n)$ be an isolated point corresponding to a non-CM elliptic curve with $\deg(x)$ odd and $j(x) \in \Q$. Then $n=2^a3^b$ for nonnegative integers $a,b$ or $j(x) \in \{-3^2\cdot 5^6/2^3, \,3^3\cdot 13/2^2\}$.
\end{thm}
\begin{proof}
Let $x=[E,P]\in X_1(n)$ be an isolated point corresponding to a non-CM elliptic curve with $\deg(x)$ odd and $j(x)\in \Q$. We fix a model for $E/\Q$. For now, we assume $j(x) \neq -3^3\cdot 13 \cdot 479^3/2^{14}$ and that $j(x)$ is not one of the two $j$-invariants in the theorem statement. By Theorem \ref{OddDegThm}, {Proposition \ref{Prop:3j}}, and \cite[Table 4]{LR}, $n=2^ap^b$ for $p \in \{3,5,7,11,13\}$ and nonnegative integers $a,b$. It suffices to consider the case where $b >0$ and $p \geq 5$. Then $a \leq 1$ by Proposition \ref{OddDegProp} {(since we have assumed for now that $j(x) \neq -3^3\cdot 13 \cdot 479^3/2^{14}$, and $3^3\cdot 13/2^2$ appears in the theorem statement)}, and $j(x)=3^3\cdot5\cdot7^5/2^7$ or $E$ corresponds to a rational point on $X_0(p)$ by Theorem \ref{OddDegThm}. For $p = 5$, we may suppose further that $E$ does not have a $\Q$-rational cyclic 25-isogeny by Proposition \ref{25lem}.

First suppose $n=p^b$, and let $f: X_1(p^b) \rightarrow X_1(p)$ be the natural map. Theorem \ref{isogenyTHM} and Lemma \ref{OneJLevel} give that $\im \rho_{E,p^{\infty}}=\pi^{-1}(\im \rho_{E,p})$, so $\deg(x)=\deg(f)\cdot\deg(f(x))$. Thus $f(x) \in X_1(p)$ is isolated by Theorem \ref{LevelLowering}. However, $X_1(p)$ has no isolated points of odd degree if $p\in \{5, 7, 11, 13\}$, as we will now demonstrate. $X_1(5)$ and $X_1(7)$ have genus 0 and thus have no isolated points. Since $X_1(11)$ and $X_1(13)$ have no non-cuspidal rational points by \cite{mazur77}, the assumption of odd degree gives $\deg(f(x)) \geq 3$. However, since $X_1(11)$ has genus 1 and $X_1(13)$ has genus 2, the Riemann-Roch space $L(f(x))$ has dimension at least 2 and $f(x)$ is not $\mathbb{P}^1$-isolated. We have reached a contradiction.

So suppose $n=2p^b$. Let $g: X_1(2p^b) \rightarrow X_1(2p)$ be the natural map. We will show that $\deg(x)=\deg(g)\cdot\deg(g(x))$. As above, Theorem \ref{isogenyTHM} and Lemma \ref{OneJLevel} show $x$ maps to a point $x'=[E,2P] \in X_1(p^b)$ such that $\deg(x')=\deg(f)\cdot\deg(f(x'))$. By Proposition \ref{prop:Degree}, we have 
\[
p^{2b-2} \mid [\Q(x):\Q(f(x'))]=[\Q(x):\Q(h(g(x)))],
\]
where $h: X_1(2p) \rightarrow X_1(p)$. Since $\deg(h)=3$ by Proposition \ref{prop:Degree}, it follows that $p^{2b-2} \mid [\Q(x):\Q(g(x))]$. Since $[\Q(x):\Q(g(x))] \leq \deg(g)=p^{2b-2}$, it follows that $[\Q(x):\Q(g(x))]=\deg(g)$, or that $\deg(x)=\deg(g)\cdot \deg(g(x))$. Thus $g(x)\in X_1(2p)$ is isolated by Theorem \ref{LevelLowering}. We will reach a contradiction by considering each prime separately.
\begin{enumerate}
\item Suppose $p = 5$. Then $X_1(10)$ has genus 0 and thus has no isolated points.
\item Suppose $p=7$. Then $X_1(14)$ has genus 1. By Mazur \cite{mazur77}, there are no non-cuspidal rational points on $X_1(14)$, so the assumption of odd degree forces $\deg(g(x)) \geq 3$. Thus the Riemann-Roch space $L(g(x))$ has dimension at least 3, and so $g(x)$ is not $\mathbb{P}^1$-isolated.
\item Suppose $p=11$. Since $E$ corresponds to a rational point on $X_0(11)$ and is non-CM, we have $j(E)=-11^2$ or $j(E)=-11\cdot 131^3$ by \cite[Table 4]{LR}. By computing division polynomials associated to a fixed model of $E/\Q$ for each $j$-invariant, we find $\deg(g(x))\geq15$. Since $X_1(22)$ has genus 6, the Riemann-Roch space $L(g(x))$ has dimension at least 10, and so $g(x)$ is not $\mathbb{P}^1$-isolated. 
\item Suppose $p=13$. Since $\deg(g(x))$ is odd, the classification of images of mod 13 Galois representations for elliptic curves over $\Q$ implies $g(x)$ maps to a point of degree 3 or 39 on $X_1(13)$. See for example \cite[Tables 1 \& 2]{GJN}. If it is degree 39, then $g(x) \in X_1(26)$ has degree at least 39. But $X_1(26)$ has genus 10, and so the Riemann-Roch Theorem shows $g(x)$ is not $\mathbb{P}^1$-isolated. So $g(x)$ must map to a point of degree 3 on $X_1(13)$. Since there are no degree 3 points on $X_1(26)$ associated to elliptic curves with rational $j$-invariant by Theorem 1.3 in \cite{Guzvic}, we have $\deg(g(x))=9$. But then if $h: X_1(26) \rightarrow X_1(13)$, we have $\deg(g(x))=\deg(h)\cdot \deg(h(g(x))$, and so $h(g(x))\in X_1(13)$ is isolated by Theorem \ref{LevelLowering}. As in the second paragraph, no such isolated point exists.
\end{enumerate}
In each case, we arrive at a contradiction.

If $j(x)= -3^3\cdot 13 \cdot 479^3/2^{14}$, then $E$ corresponds to a rational point on $X_0(p)$ only if $p=7$. Thus Theorem \ref{OddDegThm} shows that $n=2^a$, $7^b$, $2\cdot 7^b$, or $2^2 \cdot 7^b$ for $b>0$. The first cases follow as above, so it remains to consider the case when $n=2^2 \cdot 7^b$. We will show that $x$ maps to an isolated point on $X_1(28)$. Let $g:X_1(4\cdot 7^b) \rightarrow X_1(4 \cdot 7)$ and $h: X_1(4\cdot 7) \rightarrow X_1(7)$ be the natural maps. Since the $x$-coordinate of a point of order 4 satisfies a polynomial of degree 6, the degree of $\Q(g(x))$ over $\Q(h(g(x))$ is not divisible by 7. Then as in the third paragraph of this proof, 
\[
[\Q(x):\Q(g(x))]=7^{2b-2}=\deg(g).
\]
By Theorem \ref{LevelLowering}, $g(x)$ is isolated. Computing division polynomials shows that $\deg(g(x))=63$. Note $X_1(28)$ has genus 10. By the Riemann-Roch Theorem, the dimension of $L(g(x))$ is at least 54 and $g(x)$ is not $\mathbb{P}^1$-isolated.
\end{proof}

\subsection{Isolated points on $X_1(2^a3^b)$}

To study isolated points on $X_{1}(2^{a} 3^{b})$ we rely on the results of \cite{3adicimages}, which give a complete classification of the image of the $3$-adic
Galois representation for non-CM elliptic curves $E/\Q$ with a rational $3$-isogeny. The only cases that arise are parametrized by genus $0$ modular curves,
and Sutherland and Zywina \cite{SutherlandZywina} exhibit all such
subgroups containing $-I$. (Note for any $E/\Q$ there exists a twist $E'/\Q$ such that $-I \in \im \rho_{E'}$, and the choice of model does not effect the degree of a point on $X_1(N)$.) A table giving a list of these images
appears in the appendix.

\begin{prop} \label{ReductionProp}
  Let $x \in X_1(2^a3^b)$ be an isolated point corresponding to a non-CM elliptic curve with $\deg(x)$ odd and $j(x) \in \Q$. Then $x$ maps to an isolated
  point on one of $X_{1}(54)$ or $X_{1}(162)$.
\end{prop}
\begin{proof}
  Let $x=[E,P] \in X_1(2^a3^b)$ be an isolated point corresponding to a non-CM elliptic curve with $\deg(x)$ odd and $j(x) \in \Q$. We fix a model of $E/\Q$. If $b=0$, {then $x$ is not isolated as in the proof of Proposition \ref{25lem}.} Next, suppose $a=0$, and let $3^d$ be the level of the 3-adic Galois representation associated to $E$. Then $\deg(x)=\deg(f)\deg(f(x))$, where $f:X_1(3^b) \rightarrow X_1(\gcd(3^d, 3^b))$ is the natural map; see Proposition 5.8 in \cite{BELOV}. {By Lemma~\ref{OddDegLem}, $E$ has a rational $3$-isogeny and from the classification in \cite{3adicimages}, $f(x) \in X_1(3^{d'})$ is isolated by Theorem \ref{LevelLowering} for some $d' \leq 3$.} If $d' \leq 2$, then we have reached a contradiction since $X_1(3^{d'})$ has genus 0, so suppose $d'=3$. Note this is only possible if the image of the 3-adic Galois representation associated to $E$ has level 27,
  which implies that it is $27A^{0}\mhyphen 27a$. By looking at orbit sizes of points of order 9 and points of order 27 {(see Appendix)}, we see that in fact $f(x)$ will again map to an isolated point on $X_1(9)$, which is a contradiction.

Thus we may assume $a, b >0$, and $E$ has a rational $3$-isogeny by Lemma \ref{OddDegLem}. By the classification in \cite{3adicimages}, the 3-adic Galois representation has level $3^d$ for $d \in \{1, 2, 3\}$. Then by Proposition 6.1 in \cite{BELOV}, 
\[
\im \rho_{E, 6^{\infty}} = \pi^{-1}(\im \rho_{E, 2^{\alpha}3^{\beta}})
\]
where $\beta \leq d+1$. Let $g: X_1(2^a3^b) \rightarrow X_1(\gcd(2^a3^b, 2^{\alpha}3^{\beta}))$ be the natural map. By Proposition 5.8 in \cite{BELOV}, 
\[
\deg(x)=\deg(g)\deg(g(x)),
\]
and so $x$ maps to an isolated point on $X_1(\gcd(2^a3^b, 2^{\alpha}3^{\beta}))$ by Theorem \ref{LevelLowering}. Since $a \leq 2$ by Theorem \ref{OddDegThm}, and $\beta \leq 4$, it follows that $x$ maps to an isolated point on $X_1(2^m3^n)$ for $m \leq 2$ and $n \leq 4$. We have already shown we cannot have $m=0$ or $n=0$, so after removing curves of genus 0 we are left with
\begin{align*}
X_1(2\cdot3^2), \, X_1(2\cdot3^3), \, X_1(2\cdot3^4)\\
X_1(2^2\cdot3^2), \, X_1(2^2\cdot3^3), \, {\text{and }} X_1(2^2\cdot3^4){.}
\end{align*}
{Now} $X_1(18)$ has no non-cuspidal points of degree 1 by \cite{mazur}, so any point of odd degree must have degree at least 3. Since $X_1(18)$ has genus 2, the Riemann-Roch space $L(g(x))$ has dimension at least 2 and so $g(x)$ is not $\mathbb{P}^1$-isolated. It remains to rule out curves of the form $X_1(2^2 \cdot 3^n)$. Note that if $g(x) \in X_1(2^2\cdot 3^n)$ is of odd degree, then its image on $X_1(4)$ has odd degree.

First suppose $E$ does not have a rational point of order 2. Then the only way $E$ can correspond to a point on $X_1(4)$ of odd degree is if $E$ gives a rational point on X20 by \cite{RouseDZB}. However, the fiber product of $X_0(3)$ and X20 has non-cuspidal points corresponding only to $j=3^{2} \cdot 23^{3}/2^{6}$  and $j=-3^{3} \cdot 11^{3}/2^{2}$ \cite[Prop. 6]{DanielsGJ20}. {Using the classification of 3-adic images in \cite{3adicimages}, we confirm the 3-adic Galois representation associated to an elliptic curve with each of these $j$-invariants has level 3. Thus it suffices to rule out isolated points on $X_1(36)$ corresponding to these $j$-invariants. Note $X_1(36)$ has genus 17. By computing division polynomials, we see that in either case the degree of a point on $X_1(36)$ is at least 27, which means the dimension of the associated Riemann-Roch space is at least 11 and the point is not isolated.}

Suppose $E$ has a rational point of order 2. Since $\deg(X_1(4) \rightarrow X_1(2))$ has degree 2, $x$ corresponds to a point of odd degree on $X_1(4)$ only if it has degree 1. We now consider the possible images of the 3-adic Galois representation associated to $E$ in the case where it has a 4-isogeny and a 3-isogeny. It suffices to consider only those subgroups containing $-I$. Then:
\begin{itemize}
\item the $3$-adic image cannot be contained in $9B^{0}\mhyphen9a$, since that would imply $E$ had a rational cyclic 36-isogeny, contradicting Theorem \ref{IsogClassification}.
\item the $3$-adic image cannot be contained in $3D^{0}\mhyphen3a$, since that would imply $E$ had a 3-isogeny and an independent 12-isogeny. This cannot occur. See \cite[Theorem 2]{kenku}.
\item the $3$-adic image cannot be contained in $9C^{0}\mhyphen9a$ because the fiber product of $X_0(4)$ and $X_{9C^{0}\mhyphen 9a}$ covers the fiber product of $X_0(2)$ and $X_{9C^{0} \mhyphen 9a}$ and that curve has no non-cuspidal, non-CM rational points. It is genus 2 with 5 rational points: 3 cusps, $j = 0$ and $j = 2^{4} \cdot 3^{3} \cdot 5^{3}$. \href{http://users.wfu.edu/rouseja/isolated/x02andX9C.txt}{Code is available at the website of the third author.}
\end{itemize}

Thus the image of the 3-adic Galois representation associated to $E$ must be $3B^{0}\mhyphen3a$. In particular, it has level 3. This is the case where $\beta \leq 2$, so $x$ maps to an isolated point on $X_1(4 \cdot 3^2)$ of degree at least 9. We consider two cases.
\begin{enumerate}
\item Suppose $x$ lies above a point of degree 1 on $X_1(3)$. Since $x$ also lies above a point of degree 1 on $X_1(4)$, then the assumption that $x$ has odd degree means it corresponds to a point of degree 1 on $X_1(12)$. Then $x$ corresponds to a point of degree 9 on $X_1(36)$, and since $\deg(X_1(36) \rightarrow X_1(12))=9$, by Theorem \ref{LevelLowering} $x$ maps to an isolated point on $X_1(12)$. This is a contradiction since $X_1(12)$ has genus 0.
\item Suppose $x$ lies above a point of degree 3 on $X_1(3)$. Then $x$ corresponds to a point of degree at least 27 on $X_1(36)$. Since $X_1(36)$ has genus 17, the associated Riemann-Roch space has dimension at least 11 and the point on $X_1(36)$ is not isolated. 
\end{enumerate}

In every case, we have reached a contradiction, so we are left with the curves $X_1(2\cdot3^3)$ and $X_1(2\cdot3^4)$, as in the theorem statement.
\end{proof}

\subsection{Isolated points on $X_1(54)$ and $X_1(162)$}
\begin{prop} \label{54prop}
There are no odd degree, non-cuspidal, non-CM isolated points $x$ on $X_{1}(54)$ or $X_{1}(162)$ with $j(x) \in \Q$. 
\end{prop}
\begin{proof}
If $E$ is an elliptic curve and $j(x) \in \Q$, then the degree of $\Q(E[3^{k}])/\Q(j)$ is a divisor of $|\GL_{2}(\Z/3^{k} \Z)| = 2^{4} \cdot 3^{4k-3}$. This implies that if $x$ is an odd degree point with $j(x) \in \Q$, then the degree of $x$ is a power of $3$.

First suppose that $x=[E,P]$ is a non-cuspidal, non-CM isolated point on
$X_{1}(162)$, and fix a model of $E/\Q$. We will show that the image of $x$ on $X_{1}(54)$ is isolated. As in the proof of Proposition  \ref{ReductionProp}, if the level of the $3$-adic Galois representation associated to $E$ is $3^{d}$ for $d \leq 3$, the level of the $6$-adic Galois
representation is $2^{\alpha} 3^{\beta}$, where $\beta \leq d+1$. Thus $x$ maps to an isolated point on $X_1(\gcd(2^a3^b, 2^{\alpha}3^{\beta}))$. Proposition \ref{ReductionProp} shows $d \neq 1$, and if $d=2$, then $x$ would map to an isolated point on $X_1(54)$, as desired. Thus we may assume $d=3$. This implies the $3$-adic image is equal to $27A^{0}\mhyphen27a$ (up to $\pm I$). The fiber product $X_{27A^{0}\mhyphen27a} \times_{X_{0}(1)}
X_{0}(2)$ is the genus $2$ curve $y^{2} = x^{6} + 10x^{3} + 1$ whose
Jacobian has rank zero and has exactly four rational points (all of
them cusps). \href{http://users.wfu.edu/rouseja/isolated/x02andX27A.txt}{For details of the computation, see the website of the third author.} It follows that $E$ does not have a rational $2$-isogeny. Let $G$ be the image of the mod $162$ Galois representation attached
to $E$. Let $\pi_{1} : G \to \GL_{2}(\Z/81 \Z)$ and $\pi_{2} : G \to \GL_{2}(\Z/2 \Z)$ be the natural reduction maps. From the
classification of the $3$-adic representation, the image of $\pi_{1}$
contains all matrices $\equiv I \pmod{27}$. Let $H$ be the preimage
under $\pi_{2}$ of a subgroup of $\GL_{2}(\Z/2\Z)$ of index $3$. Then
$H$ has index $3$ in $G$ and so $\pi_{1}(H)$ is either equal
to the mod $81$ image of Galois or an index $3$ subgroup
thereof. However, every maximal subgroup of $27A^{0}\mhyphen27a$ has level $27$ and this means
that $\pi_{1}(H)$ contains all matrices congruent to the identity mod
$27$. If $g \in H$ is congruent to the identity modulo $27$
then $g^{2}$ is congruent to the identity modulo $54$ and from this we
see that $G$ contains all matrices congruent to the identity modulo
$54$. This implies that the degree of $x$ on $X_{1}(162)$ is
$\deg(X_{1}(162) \to X_{1}(54))$ times the degree of the image of $x$
on $X_{1}(54)$ and so the image of $x$ on $X_{1}(54)$ is isolated by
Theorem~\ref{LevelLowering}.

For the remainder of the proof, we will assume that $x$ is an odd degree,
non-cuspidal, non-CM isolated point on $X_{1}(54)$ with $j(x) \in \Q$. We fix a model of $E/\Q$.

Suppose that $E$ has a rational point of order $2$. The fact that $X_{1}(18)$ has no odd degree isolated points with rational $j$-invariant implies that the level of the $3$-adic Galois representation must be $27$ but as mentioned above, the fiber product $X_{27A^{0}\mhyphen 27a} \times_{X_{0}(1)} X_{0}(2)$ has no non-cuspidal rational points, which is a contradiction.

Next, suppose that the elliptic curve $E$ has no rational point of order $2$
and that $\Q(E[2]) \cap \Q(E[27])$ is either $\Q$ or a quadratic extension of $\Q$. It follows from this that the degree of $x$ on $X_{1}(54)$ is three times
the degree of $x$ on $X_{1}(27)$ and by Theorem~\ref{LevelLowering},
the image of $x$ on $X_{1}(27)$ must be isolated. From the $3$-adic classification, this does not occur. 

Next, suppose that $\Q(E[2]) \cap \Q(E[27])$ is a cyclic cubic
extension of $\Q$. This implies that $E$ has square discriminant $\Delta(E)$. A straightforward computation shows that $(j(E) - 1728) \Delta(E)$ is a square, from which it follows that $j(E)=1728+t^2$ for some $t \in \Q$. There are infinitely many elliptic curves with square discriminant and with $3$-adic image contained in $9C^{0}\mhyphen 9a$, but the modular curves parametrizing elliptic curves with square discriminant and a $9$-isogeny, or square discriminant and a pair of
two independent $3$-isogenies both have genus $1$ and are isomorphic
to $y^{2} = x^{3} - 27$. This elliptic curve has two rational points
and in both cases, these rational points are cusps. (See \href{http://users.wfu.edu/rouseja/isolated/X9B0squaredisc.txt}{here} and \href{http://users.wfu.edu/rouseja/isolated/X3D0squaredisc.txt}{here} for the code.) It remains to
consider subgroups of $9C^{0} \mhyphen 9a$, and these are $9J^{0} \mhyphen 9a$, $9J^{0} \mhyphen9b$
and $9J^{0} \mhyphen9c$. The latter two would give rise to
points on $X_{1}(54)$ of degree 81 or higher, so it suffices to consider the fiber product of $X_{9J^{0} \mhyphen 9a}$ and the curve $j = 1728+t^{2}$. This curve has genus $2$ and is isomorphic to $y^{2} = x^{5} + 4x^{4} + 3x^{3} - x^{2} - x$. The
Jacobian has rank zero and the curve has precisely three rational
points, all of which are cusps. (See \href{http://users.wfu.edu/rouseja/isolated/X9J09asquaredisc.txt}{here} for details.) Therefore, this case does not occur.

Finally, suppose $\Q(E[2]) \cap \Q(E[27])$ is an $S_{3}$-extension of
$\Q$.  Since $x$ must give rise to a point of degree $\leq 27$ on
$X_{1}(54)$, it follows that $E$ must have a point of order $9$ in
degree $1$ or $3$. This forces the $3$-adic image to equal (up to $\pm
I$) $9B^{0} \mhyphen 9a$, $9H^{0} \mhyphen 9b$, $9I^{0} \mhyphen 9a$, $9I^{0} \mhyphen 9b$, $9I^{0} \mhyphen 9c$,
$9J^{0} \mhyphen 9a$, or $27A^{0} \mhyphen 27a$. For each of these cases, let $L
\subseteq \GL_{2}(\Z/54 \Z)$ be the preimage of the mod $27$ image of
Galois. We enumerate index $6$ subgroups $K \subseteq L$ with the
property that the kernel of $\pi_{1} : K \to \GL_{2}(\Z/2\Z)$ has
index $6$ and contains the kernel of $\pi_{2} : K \to
\GL_{2}(\Z/27\Z)$. The elliptic curve $E$ must correspond to a
rational point on one of the curves $X_{K}$. Any $K$ which forces $E$
to give a degree $3$ point on $X_{1}(18)$ and a degree $27$ point on
$X_{1}(54)$ can be ruled out, since the image on $X_{1}(18)$ must be
isolated. In the end we find {at most} one choice of $K$ for each of the
$3$-adic images mentioned above, but all of the modular curves $X_K$ are
covers of the curve for the subgroup $9B^{0} \mhyphen 9a$. (See the \href{http://users.wfu.edu/rouseja/isolated/S3entanglement.txt}{code at the website of the third author for details}. Also note that
the $3$-adic image being contained in $9B^{0} \mhyphen 9a$ is equivalent to saying that $E$ has a $9$-isogeny defined over $\Q$.) For the rest of the
proof, $K$ will denote this particular subgroup of $\GL_{2}(\Z/54
\Z)$. The curve $X_{K}$ has genus $4$, and group theoretic computations indicate
that $X_{K}$ has a map to a genus $1$ curve. Also, there is an element $\vec{v}$
of order $2$ in $(\Z/54 \Z)^{2}$ so that the intersection
of $K$ with the stabilizer of $\vec{v}$ 
is {contained in} $\left\{ \begin{bmatrix} a & b \\ c & d \end{bmatrix} \in
\GL_{2}(\Z/54 \Z) : c \equiv 0 \pmod{54} \right\}$.  This implies that
if $E$ is an elliptic curve with no rational $2$ torsion and $\im \rho_{E,54} \subseteq K$, then the cubic subfield in
$\Q(E[2])$ is the same field over which the curve $E$ acquires a
cyclic isogeny of degree $27$.

According to Magma's small modular curves database, the map from
$x_{9} : X_{0}(9) \to X_{0}(1)$ is given by
\[
j = \frac{(x_{9} + 9)^{3} (x_{9}^{3} + 243x_{9}^{2} + 2187x_{9} + 6561)^{3}}{x_{9}^{9} (x_{9}^{2} + 9x_{9} + 27)}.
\]
The curve $X_{0}(27)$ has equation $y^{2} + y = x^{3} - 7$. Define $\phi : X_{0}(27) \to X_{0}(9)$ by $\phi(x,y) = -3 + (y+5)/x$. Then $x_{9} \circ \phi$ is the map from $X_{0}(27)$ to the $j$-line.
We wish to represent $X_{0}(27)$ as a degree $3$ cover of $X_{0}(9)$ via this map,
and a Groebner basis computation shows that if $x_{9} \in \Q$, the
$x$-coordinate of a preimage of $x_{9}$ under $\phi$ satisfies
\[
  x^{3} - (x_{9}^{2} + 6x_{9} + 9) x^{2} + (9x_{9} + 27) x - 27 = 0.
\]
We make a change of variables, setting $t = 3x - (x_{9} + 3)^{2}$ and obtain
\[
p_{1}(t,x_{9}) = t^{3} - (3x_{9}^{4} + 36x_{9}^{3} 162 x_{9}^{2} + 243x_{9})t -
(2x_{9}^{6} + 36x_{9}^{5} + 270x_{9}^{4} + 999x_{9}^{3} + 1701x_{9}^{2} + 729x_{9}) = 0.
\]
This makes the coefficient of $t^{2}$ equal to $0$. Using an equation for $X_{0}(2)$, one finds that the degree $3$ subfield of $\Q(E[2])$ is given by
$p_{2}(t,x_{9}) = t^{3} - jt - 16j = 0$. We wish to determine the values of
$x_{9}$ for which these two polynomials define isomorphic degree $3$ extensions.

Given two irreducible cubic polynomials in $\Q(x)[t]$ with $p_{1}$
having roots $t_{1}$, $t_{2}$ and $t_{3}$ (in an algebraic closure of
$\Q(x)$) and $p_{2}$ having roots $t_{4}$, $t_{5}$ and $t_{6}$,
suppose $x \in \Q$ has the property that $p_{1}$ and $p_{2}$ define
isomorphic degree $3$ extensions of $\Q$.  Let $L$ be this degree
$3$ extension, and $M$ be its Galois closure.  Fix an ordering of the
roots so that if $\sigma \in \Gal(M/\Q)$ sends $t_{1}$ to $t_{2}$
and $t_{2}$ to $t_{3}$, it sends $t_{4}$ to $t_{5}$ and $t_{5}$ to
$t_{6}$ and view $\Gal(M/\Q)$ as a subgroup of $S_{6}$. It follows
that $\theta_{1} = t_{1} t_{4} + t_{2} t_{5} + t_{3} t_{6} = {\rm
  tr}_{L/\Q}(t_{1} t_{4}) \in \Q$. Let $\theta_{2}$,
$\theta_{3}$, $\theta_{4}$, $\theta_{5}$ and $\theta_{6}$ be the other
elements in the orbit of $\theta_{1}$ under the action of $S_{3}
\times S_{3}$. It follows that $f(t,x) = \prod_{i=1}^{6} (t -
\theta_{i}) \in \Q[t]$ and has a rational root.

A linear algebra calculation carried out in Magma (see \href{http://users.wfu.edu/rouseja/isolated/deg3iso.txt}{here} for details) shows that
if $p_{1}(t,x) = t^{3} + A_{2} t + A_{3}$ and $p_{2}(t,x) = t^{3} + B_{2} t + B_{3}$,
then
\[
f(t,x) = t^{6} - 6 A_{2} B_{2} t^{4} - 27 A_{3} B_{3} t^{3} + 9 A_{2}^{2} B_{2}^{2}
t^{2} + 81 A_{2} A_{3} B_{2} B_{3} t - 4A_{2}^{3} B_{2}^{3} - 27 A_{2}^{3} B_{3}^{2}
- 27 A_{3}^{2} B_{2}^{3}.
\]
We apply this to the two polynomials $p_{1}(t,x_{9})$ and
$p_{2}(t,x_{9})$ stated above. This gives rise to an equation involving
$t$ and $x_{9}$ which has degree $41$ in $x_{9}$ and degree $6$
in $t$. We wish to find all of the rational points on the curve defined by this equation. Using the methods in van Hoeij and Novocin's preprint \cite{vanHoeijNovocin}, we are able to find a much
simpler polynomial that defines the same function field. We find the
polynomial
\[
X : t^{6} + (-2x_{9}^{3} - 18x_{9}^{2} - 54x_{9}) t^{3} + x_{9}^{6} + 18x_{9}^{5}
+ 135x_{9}^{4} + 513x_{9}^{3} + 972x_{9}^{2} + 729x_{9} = 0.
\]
The map $(t,x_{9}) \mapsto (t^{3},x_{9})$ is clearly a map to the curve
\[
Y : y^{2} + (-2x_{9}^{3} - 18x_{9}^{2} - 54x_{9}) y + x_{9}^{6} + 18x_{9}^{5} + 135x_{9}^{4} + 513x_{9}^{3} + 972x_{9}^{2} + 729x_{9} = 0
\]
This curve $Y$ has genus $1$ and is isomorphic to $y^{2} = x^{3} + 1$. This elliptic curve has rank zero and Mordell-Weil group $\Z/6\Z$. The six rational points on $Y$ are $(-324 : -9 : 1)$, $(0 : 0 : 1)$, $(1 : 0 : 0)$,
$(-162 : -9 : 1)$, $(0 : -3 : 1)$, and $(-54 : -3 : 1)$. Of these six points,
two are rational cusps, two have image $j = 0$, and two have image $j = -2^{15} \cdot 3 \cdot 5^{3}$. Only two of the rational points on $Y$ lift to rational points on $X$, and those are the rational cusps. \href{http://users.wfu.edu/rouseja/isolated/genus4.txt}{A script documenting this computation can be found at the website of the third author.}
\end{proof}

\subsection{Non-CM isolated points of odd degree}

\begin{thm}\label{Thm: nonCM}
Let $\mathcal{I}_{odd}$ denote the set of all isolated points of odd degree on all modular curves $X_1(N)$ for $N \in \mathbb{Z}^+$. Then the non-CM $j$-invariants in $j(\mathcal{I}_{odd}) \cap \Q$ are $-3^2\cdot 5^6/2^3$ and $3^3\cdot 13/2^2$.
\end{thm}

\begin{proof}
  The fact that $j(\mathcal{I}_{odd}) \cap \Q \subseteq \{-3^2\cdot 5^6/2^3, 3^3\cdot 13/2^2\}$ follows from Theorem \ref{IsolPtsThm}, Proposition \ref{ReductionProp}, and Proposition \ref{54prop}. It remains to show that these two $j$-invariants correspond to isolated points of odd degree. By work of Najman \cite{najman16}, there is an isolated (in fact, sporadic) point $x \in X_1(21)$ with $\deg(x)=3$ and $j(x)=-3^2\cdot 5^6/2^3$. We have also identified a degree 9 point $x \in X_1(28)$ corresponding to the elliptic curve $E$ with LMFDB label 338.e2. Since the Jacobian of $X_1(28)$ has rank 0 \cite[Lemma 1]{DerickxVanHoeij}, it suffices to show $x$ is $\mathbb{P}^1$-isolated. We use the model of $X_{1}(28)$ computed by Sutherland \cite{SutherlandX1} (see Table 6). The universal elliptic curve has the form $E_{u} : y^{2} + xy + uy = x^{3} + ux^{2}$ for some
  $u \in \Q(X_{1}(28))$. We first find the choices of $u$ in the degree $9$
  number field $\Q(x)$ for which $j(E_{u}) = 3^{3} \cdot 13/2^{2}$. There are two such,
  but only one gives points in the desired degree $9$ number field. In the end, we find $6$ points on $X_{1}(28)$ over the desired number field that are interchanged by diamond automorphisms and choose one of them to create a degree $9$ divisor $D$
  over $\Q$ on $X_{1}(28)$. Since the natural reduction of a principal divisor is principal over any prime of good reduction \cite[Thm. 9.5.1]{BLR90}, it suffices to show that the Riemann-Roch space $L(\tilde{D})$ over $\mathbb{F}_{11}$ is one-dimensional. This can be verified in Magma; \href{http://users.wfu.edu/rouseja/isolated/X128.txt}{see the website of the third author for the Magma code used.} Thus there are no non-constant functions $f : X_{1}(28) \to \mathbb{P}^{1}$ over $\Q$ with poles only at $D$ and so the degree $9$ point on $X_{1}(28)$ is isolated. 
\end{proof}

\section{The CM Case}

{In this section we show that any CM $j$-invariant in $j(\mathcal{I}_{odd}) \cap \Q$ belongs to the set  
\[
\{-2^{18}3^35^3, \, -2^{15}3^35^311^3, \, -2^{18}3^35^323^329^3\},
\] completing the proof of Theorem \ref{MainThm}. These are the elliptic curves with CM by the orders of discriminant $-43, -67, -163$, respectively.} Our results follow from work of the first author and Clark \cite{BC1}, \cite{BC2}.

\subsection{Preliminaries on Cartan orbits} We first recall the necessary ingredients from \cite[$\S7$]{BC1}. Let $\OO$ be an order in an imaginary quadratic field and let $N$ be a positive integer. If $P \in \OO/N\OO$ is a point of order $N$ {(which by $\S2.6$ corresponds to a point of order $N$ on an $\OO$-CM elliptic curve)}, then define $M_P \coloneqq  \{xP \mid x \in \OO\}$ to be the $\OO$-submodule of $\OO/N\OO$ generated by $P$ and $I_P \coloneqq \{ x \in \OO \mid xP =0\}$. There is a canonical $\OO$-module isomorphism
\[
M_P \cong \OO/I_P
\]
defined by $P \mapsto 1 + I_P$. We may use this isomorphism to determine the size of the $(\OO/N\OO)^{\times}$-orbit on $P$. Recall we denote $(\OO/N\OO)^{\times}$ by $C_N(\OO)$.

\begin{lem} \label{orbit1}
Let $p$ be an odd prime, and let $\OO$ be an imaginary quadratic order of discriminant $\Delta$ such that $\leg{\Delta}{p} =0$. Let $P \in \OO/p^a\OO$ be a point of order $p^a$. Then as abelian groups, $M_P \cong \Z/p^a\Z \times \Z/p^{b}\Z$ for some integers $0 \leq b \leq a$, and the $C_{p^a}(\OO)$-orbit on $P$ has size $p^{a+b-1}(p-1)$.
\end{lem}

\begin{proof}
The fact that $M_P \cong_{\Z} \Z/p^a\Z \times \Z/p^{b}\Z$ for some $0 \leq b \leq a$ is Lemma 7.5 in \cite{BC1}. Since $\leg{\Delta}{p} =0$, $\OO/I_P$ is local with residue field $\Z/p\Z$. Thus 
\[\# (\OO/I_P)^{\times}=\#\OO/I_P-\frac{\#\OO/I_P}{p}=p^{a+b-1}(p-1).
\] By \cite[Lemma 7.4]{BC1}, the size of the $C_N(\OO)$-orbit on $P$ is equal to the size of $(\OO/I_P)^{\times}$.
\end{proof}

\begin{lem} \label{orbit2}
Let $p$ be an odd prime, and let $\OO$ be an imaginary quadratic order of discriminant $\Delta$ such that $\leg{\Delta}{p}\neq 1$. Let $P \in \OO/p^a\OO$ be a point of order $p^a$ for $a \in \Z^+$. Suppose for some integer $0 \leq m \leq a$ the $C_{p^m}(\OO)$-orbit on $p^{a-m}P  \in \OO/p^m\OO$ has size greater than $\varphi(p^m)$. Then the size of the $C_{p^a}(\OO)$-orbit on $P$ is equal to $p^{2(a-m)}$ times the size of the $C_{p^m}(\OO)$-orbit on $p^{a-m}P$.
\end{lem}
\begin{proof} If $\leg{\Delta}{p}=-1$, this follows from \cite[Theorem 7.8]{BC1}, so henceforth we may assume $\leg{\Delta}{p}=0$. Following \cite[$\S7$D]{BC1}, we observe that $x \mapsto p^{a-m}x$ gives an $\OO$-module isomorphism
\[
\OO/p^m\OO \rightarrow p^{a-m}\OO/ p^a \OO.
\]
This allows us to view $\OO/p^m\OO$ as an $\OO$-submodule of $\OO/ p^a \OO$. Since we have assumed the $C_{p^m}(\OO)$-orbit on $p^{a-m}P  \in \OO/p^m\OO$ has size greater than $\varphi(p^m)$, Lemma \ref{orbit1} shows
\begin{align*}
M_P \cong_{\Z} \Z/p^a\Z \times \Z/p^{b}\Z,\\
 M_{p^{a-m}P} \cong_{\Z} \Z/p^m\Z \times \Z/p^{b'}\Z
\end{align*}
for some $0 \leq b \leq a$ and $1 \leq b' \leq m$. Since $p^{a-m} M_P=M_{p^{a-m}P}$, we see that $b=b'+a-m$. Another application of Lemma \ref{orbit1} shows that $P$ lies in a $C_{p^{a}}(\OO)$-orbit of size $p^{2(a-m)}$ times the size of the $C_{p^m}(\OO)$-orbit  on $p^{a-m}P$. 
\end{proof}

\subsection{CM version of Theorem \ref{isogenyTHM}} The following theorem shows that, as in the case of non-CM elliptic curves over $\Q$, points on $X_1(p^a)$ corresponding to a CM elliptic curve $E$ with a rational cyclic $p$-isogeny over $\Q(j(E))$ often arise in largest possible degree allowed by the isogenies.
For relevant background information on CM elliptic curves, see Section \ref{CMsection}.

\begin{prop} \label{CMprop}
Let $p$ be an odd prime and let $E$ be a $K$-CM elliptic curve. Define $m$ to be the maximum integer such that there exists $y \in X_0(p^m)(\Q(j(E))$ with $j(y)=j(E)$.\footnote{For the explicit values of $m$, see Propositions 6.4 and 6.8 in \cite{BC2}. These can mostly be deduced from \cite{kwon99}.}  If $m \geq 1$ and $\leg{\Delta_K}{p} \neq 1$, then for any integer $a >m$ and any point $x \in X_1(p^a)$ with $j(x)=j(E)$, we have
\[
\deg(x)=\deg(f(x))\cdot \deg(f),
\]
where $f: X_1(p^a) \rightarrow X_1(p^m)$ is the natural map.
\end{prop}

\begin{proof} Let $x=[E,P] \in X_1(p^a)$. The assumption that $m \geq 1$ means there exists a model of $E/\Q(j(E))$ with a rational cyclic $p$-isogeny. Since $p$ is odd, $\leg{\Delta}{p}=0$ by Proposition 6.8 of \cite{BC2}. Furthermore, if $\Delta=-3$, then $m=2$; see, for example, \cite[Corollary 5.11]{BC2}.

As in the proof of Lemma \ref{orbit2}, we may identify $P$ with an element of $\OO/p^a\OO$ of order $p^a$ and $p^{a-m}P$ with an element of $\OO/p^m\OO$ of order $p^m$. Suppose first that the $C_{p^m}(\OO)$-orbit on $p^{a-m}P$ has size greater than $\varphi(p^m)$. Thus by \cite[Lemma 7.6]{BC1} and Lemma \ref{orbit2} we have the lower bound
\[
[K(\mathfrak{h}(P)):K(\mathfrak{h}(p^{a-m}P))]=p^{2(a-m)}\leq [\Q(\mathfrak{h}(P)):\Q(\mathfrak{h}(p^{a-m}P))].
\] 
Since $p^{2(a-m)}=\deg(X_1(p^a) \rightarrow X_1(p^{m}))$, we also have the upper bound
\[
[\Q(\mathfrak{h}(P)):\Q(\mathfrak{h}(p^{a-m}P))] \leq p^{2(a-m)}
\] 
Thus equality holds, and $\deg(x)=\deg(f(x))\cdot \deg(f)$.

So suppose the $C_{p^{m}}(\OO)$-orbit of $p^{a-m}P$ has size less than or equal to $\varphi(p^{m})$.\footnote{In fact, this implies the orbit size must be exactly $\varphi(p^{m})$ since $C_{p^m}(\OO)$ contains all scalar matrices.} Suppose for the sake of contradiction that the size of the $C_{p^{a}}(\OO)$-orbit on $P$ is strictly smaller than $p^{2(a-m)}$ times the size of the $C_{p^m}(\OO)$-orbit  on $p^{a-m}P$. Then the $C_{p^{a}}(\OO)$-orbit on $P$ has size less than
\[
p^{2(a-m)}\cdot \varphi(p^m)=p^{2a-m-1}(p-1).
\]
With the values of $m$ given in \cite[Proposition 6.4]{BC2}, we find this contradicts \cite[Theorem 7.2]{BC1} since we have also assumed $\leg{\Delta_K}{p} \neq 1$. Thus the $C_{p^{a}}(\OO)$-orbit on $P$ is equal to $p^{2(a-m)}$ times the size of the $C_{p^m}(\OO)$-orbit  on $p^{a-m}P$, and the argument follows as before. 
\end{proof}

\begin{rmk}
The statement of Proposition \ref{CMprop} does not hold if $\leg{\Delta_K}{p} =1$. In this case, for such a $K$-CM elliptic curve $E$, there exists $y' \in X_1(p^{M})(K(j(E))$ with $j(y)=j(E)$ for all $M \in \Z^+$. See Proposition 6.4 in \cite{BC2}. These extra isogenies picked up over $K(j(E))$ prevent the degree condition of Proposition \ref{CMprop} from being satisfied, and they may be used to produce sporadic points associated to any CM $j$-invariant. See Theorem 7.1 in \cite{BELOV}.
\end{rmk}

\subsection{Isolated CM points of odd degree}

There are 13 CM $j$-invariants in $\Q$ corresponding to imaginary quadratic orders of discriminant 
\[
\Delta\in \{-3, -4, -7, -8, -11, -12, -16, -19, -27, -28, -43, -67, -163\}.
\] 
For most of these, we can show there is no corresponding isolated point in odd degree using the following theorem.

\begin{thm} \label{CMmapThm}
Let $x \in X_1(N)$ be an isolated point of odd degree corresponding to an elliptic curve $E$ with CM by the order in $K$ of discriminant $\Delta$. Then $K=\Q(\sqrt{-p})$ for a prime $p \equiv 3 \pmod{4}$ and $N=p^r$ or $2p^r$. Moreover, if $m$ is the maximum integer such that there exists $y \in X_0(p^m)(\Q(j(E))$ with $j(y)=j(E)$, then:
\begin{enumerate}
\item If $\leg{\Delta}{2}=-1$, then $f(x) \in X_1(\gcd(N,p^m))$ is isolated where $f: X_1(N) \rightarrow X_1(\gcd(N,p^m))$ is the natural map.
\item If $\leg{\Delta}{2}\neq-1$, then $f(x) \in X_1(\gcd(N,2p^m))$ is isolated where $f: X_1(N) \rightarrow X_1(\gcd(N,2p^m))$ is the natural map.
\end{enumerate} 
\end{thm}

\begin{rmk}
As noted above, for explicit values of $m$, see Propositions 6.4 and 6.8 in \cite{BC2}; also \cite{kwon99}.
\end{rmk}

\begin{proof}
Let $x=[E,P] \in X_1(N)$ be an isolated point of odd degree associated to an elliptic curve with CM by the order in $K$ of discriminant $\Delta$. Note there are no isolated points on $X_1(2)$ or $X_1(4)$ as they have genus 0. Thus by \cite[Cor. 9.4]{aoki95}, the assumption of odd degree implies $N=p^r$ or $2p^r$ where $K=\Q(\sqrt{-p})$ and $p \equiv 3 \pmod{4}$ is prime. If $N=p^r$, we may assume $m<r$, for otherwise the statement is clearly true. Since $\leg{\Delta_K}{p}=0$, we have $m \geq 1$ (see for example \cite[Prop. 6.4]{BC2}), so we may apply Proposition \ref{CMprop}. Then by Theorem \ref{LevelLowering}, $x$ maps to an isolated point on $ X_1(p^m)$, and the statement holds.

Next, suppose $N=2p^r$. Note we may assume $r\geq1$, and if $p=3$, we may assume $r>1$ since $X_1(6)$ has genus 0. If $\leg{\Delta}{2}=-1$, then by \cite[Lemma 7.1, Proposition 7.7]{BC1}, the size of the $C_N(\mathcal{O})$-orbit of $P$ is equal to $3$ times the size of the $C_{p^r}(\mathcal{O})$-orbit of $2P$. Lemma 7.6 of \cite{BC1} shows that 
\[
[K(j(E))(\mathfrak{h}(P)):K(j(E))]=3 \cdot [K(j(E))(\mathfrak{h}(2P)):K(j(E))].
\] Since we have assumed $[\Q(j(E))(\mathfrak{h}(P)):\Q]$ has odd degree, it follows that
\[
\deg(x)=\deg(g)\cdot \deg(g(x))
\]
where $g:X_1(2p^r) \rightarrow X_1(p^r)$ is the natural map. By Theorem \ref{LevelLowering}, $g(x) \in X_1(p^r)$ is isolated. If $m \geq r$, we are done. Otherwise the argument follows as before.

If $\leg{\Delta}{2}\neq -1$, we may assume $m<r$. Then \cite[Theorem 6.2, 6.6]{BC2} shows there is a point in $X_1(2)(\Q(j(E))$ corresponding to $E$, and the assumption that $x$ has odd degree forces $[E,p^rP] \in X_1(2)$ to have degree $[\Q(j(E):\Q]$. Thus by Proposition \ref{CMprop}, we have $\deg(x)=\deg(g(x)) \cdot \deg(g)$ where $g:X_1(2p^r) \rightarrow X_1(2p^m)$ is the natural map. By Theorem \ref{LevelLowering}, $g(x) \in X_1(2p^m)$ is isolated, as desired. \end{proof}

\begin{table}[]

\label{undefined}
\begin{tabular}{c|c|c|c|c|c}
$\Delta$ & $p$ & $m$ & genus of $X_1(p^m)$ & $d_{\Delta}$ & genus of $X_1(2p^m)$ if $\leg{\Delta}{2} \neq -1$ \\ \hline
-3       & 3   & 2   & 0                   &      3        & --                                  \\ \hline
-7       & 7   & 1   & 0                   &      3        & 1                                     \\ \hline
-11      & 11  & 1   & 1                   &     5        & --                                    \\ \hline
-12      & 3   & 1   & 0                   &    1          & 0                                      \\ \hline
-19      & 19  & 1   & 7                   &      9        & --                                     \\ \hline
-27      & 3   & 3   & 13                  &      9        & --                                     \\ \hline
-28      & 7   & 1   & 0                   &      3       & 1                                      \\ \hline
-43      & 43  & 1   & 57                  &     21         & --                                    \\ \hline
-67      & 67  & 1   &  155    &     33         & --                                     \\ \hline
-163     & 163 & 1   &  1027  &      81        & --                                   
\end{tabular}
\vspace{.5cm}
\caption{Let $m$ be as in Theorem \ref{CMmapThm} and let $d_{\Delta}$ be the least degree of a $\Delta$-CM point on $X_1(p^m)$. For values of $m$, see \cite{kwon99} and \cite[Proposition 6.4]{BC2}. The value $d_{\Delta}$ is given in \cite[Theorem 7.1]{BC2}.}
\end{table}

\begin{cor} \label{isolCM}
There are no isolated points $x \in X_1(N)$ of odd degree corresponding to an elliptic curve with CM by the order of discriminant $\Delta \in \{-3, -4, -7, -8, -11, -12, -16, -19, -27, -28 \}$.
\end{cor}

\begin{proof}
Let $x\in X_1(N)$ be an isolated point of odd degree corresponding to an elliptic curve of discriminant $\Delta \in \{-3, -4, -7, -8, -11, -12, -16, -19, -28 \}$. Since $X_1(2)$ and $X_1(3)$ have genus 0, we may assume $N>3$. By Theorem \ref{CMmapThm}, we may assume $\Delta \notin \{-4, -8, -16\}$, and $x$ maps to an isolated point $f(x)$ in $X_1(\gcd(N,p^m))$ if $\leg{\Delta}{2}=-1$ or in $X_1(\gcd(N,2p^m))$ if $\leg{\Delta}{2} \neq -1$ for $m,p$ as in the theorem statement. By Table 1, we see that the degree of $f(x)$ is larger than the genus of the curve, which means the dimension of the associated Riemann-Roch space is at least 2. Thus $f(x)$ is not $\mathbb{P}^1$-isolated and we have reached a contradiction.

{Now, let $x\in X_1(N)$ be an isolated point of odd degree corresponding to an elliptic curve with CM by the order of discriminant $\Delta=-27$. Then $j(x)=-2^{15}\cdot3\cdot5^3$, and by Theorem \ref{CMmapThm} and Table 1, $f(x) \in X_1(\gcd(N,3^3))$ is isolated where $f: X_1(N) \rightarrow X_1(\gcd(N,3^3))$ is the natural map. $X_1(3)$ and $X_1(9)$ are genus 0 and thus have no isolated points, so it suffices to show there are no isolated points of odd degree on $X_1(27)$ associated to this $j$-invariant. By computing division polynomials, we see that any point $x' \in X_1(27)$ of odd degree with $j(x')=-2^{15}\cdot3\cdot5^3$ has degree 9 or 243. Since $X_1(27)$ has genus 13, any point of degree 243 is not isolated by the Riemann-Roch Theorem, so we need only consider the point on $X_1(27)$ of degree 9.}

Since the Jacobian of $X_1(27)$ has rank 0 \cite[Lemma 1]{DerickxVanHoeij}, it suffices to show $x'$ is not $\mathbb{P}^1$-isolated. We do this by forming the associated divisor and computing its Riemann-Roch space.  First we find the Tate normal form of an elliptic curve $E(b,c)$ with $(0,0)$ of order 27. This is done by constructing a polynomial $f_{27} \in \Q[b,c]$ that vanishes when $(0,0)$ has order $27$ {on $E(b,c)$}, as in \cite[Lemma 2.4]{CompwithCM}. {Using $E(b,c)$, we find the associated point on a model of $X_1(27)$ computed by Sutherland \cite{SutherlandX1}. This allows us} to create a degree $9$ divisor $D$ over $\Q$ on $X_{1}(27)$, and {a Magma computation shows} that the Riemann-Roch space $L(D)$ over $\Q$ has dimension $3$. \href{http://users.wfu.edu/rouseja/isolated/X127.txt}{See the website of the third author for the Magma code used.}
\end{proof}

\begin{rmk} \label{finalCMjRmk}
Suppose $x\in X_1(N)$ is an isolated point of odd degree corresponding to an elliptic curve with CM by the order of discriminant $\Delta \in \{-43, -67, -163\}$. These discriminants correspond to elliptic curves with $j$-invariants $-2^{18}3^35^3$, $-2^{15}3^35^311^3$, and $-2^{18}3^35^323^329^3$, respectively. By Theorem \ref{CMmapThm}, $N=p^r$ or $2p^r$ where $p=43, 67,$ or $163$, respectively. Moreover, since $m=1$ in each case \cite{kwon99}, Theorem \ref{CMmapThm} shows $f(x) \in X_1(\gcd(N,p))$ is isolated, where $f:X_1(N) \rightarrow X_1(\gcd(N,p))$ is the natural map. Thus $-2^{18}3^35^3$, $-2^{15}3^35^311^3$, and $-2^{18}3^35^323^329^3$ are in $j(\mathcal{I}_{odd}) \cap \Q$ if and only if they correspond to an isolated point of odd degree on $X_1(p)$. In each case, the Jacobian of $X_1(p)$ has positive rank; see Proposition 6.2.1 in \cite{CES2003}. Thus to find $j(\mathcal{I}_{odd}) \cap \Q$, one must determine whether these points belong to an infinite family parametrized by a positive rank abelian subvariety of $\text{Jac}(X_1(p))$.
\end{rmk}

\appendix
\section{$3$-adic images of Galois}

{In \cite{3adicimages}, the authors determine the $3$-adic image of Galois
for every non-CM elliptic curve $E/\Q$ that has a rational $3$-isogeny. Every case that occurs arises from a genus $0$ modular curve with infinitely many rational points. The prime power level modular curves with infinitely many rational points for subgroups that contain $-I$ were determined by Sutherland and Zywina \cite{SutherlandZywina}.} The following table is an excerpt of the table from page 2 of the online supplement to \cite{SutherlandZywina} that specifies a label, the index, the level, generators, and
a map to a covering modular curve. {Here, $3B^{0}\mhyphen3a$, $3D^{0}\mhyphen3a$, $9B^{0}\mhyphen9a$, and  $9I^{0}\mhyphen9c$ denote the curves $X_0(3), X_0(3,3), X_0(9),$ and $X_1(9)$, respectively. For each $3$-adic image with level $3^{k}$, we also give the degrees on $X_{1}(3^{k})$ of each Galois orbit of points order $3^{k}$.\\

{
\begin{tabular}{lllllll}
label  & $N$ & map & covering group & $3^k$ & Orbit sizes\\
\hline
$3B^{0}\mhyphen3a$  & $3$ & $(t+3)^{3}(t+27)/t$ & $j$-line & 3 & $[1,3]$\\
$3D^{0}\mhyphen3a$  & $3$ & $729/(t^{3}-27)$ & $3B^{0}\mhyphen3a$ & 3 & $[1,1,2]$ \\
$9B^{0}\mhyphen9a$  & $9$ & $t(t^{2}+9t+27)$ & $3B^{0}\mhyphen3a$ & 9 & $[3,6,27]$ \\
$9C^{0}\mhyphen9a$  & $9$ & $t^{3}$ & $3B^{0}\mhyphen3a$ & 9 & $[9,27]$ \\
$9H^{0}\mhyphen9a$  & $9$ & $3(t^{3}+9)/t^{3}$  & $3D^{0}\mhyphen3a$ & 9 & $[9,9,18]$\\
$9H^{0}\mhyphen9b$  & $9$ & $3(t^{3} + 9t^{2} - 9t - 9)/(t^{3} - 9t^{2} - 9t + 9)$ & $3D^{0}\mhyphen3a$ & 9 & $[3,3,3,6,18]$\\
$9H^{0}\mhyphen9c$  & $9$ & $-6(t^{3}-9t)/(t^{3} + 9t^{2} - 9t - 9)$   & $3D^{0}\mhyphen3a$ & 9 & $[9,9,18]$\\
$9I^{0}\mhyphen9a$  & $9$ & $-6(t^{3}-9t)/(t^{3} - 3t^{2} - 9t + 3)$ & $9B^{0}\mhyphen9a$ & 9 & $[3,6,9,9,9]$\\
$9I^{0}\mhyphen9b$  & $9$ & $-3(t^{3} + 9t^{2} - 9t - 9)/(t^{3} + 3t^{2} - 9t - 3)$ & $9B^{0}\mhyphen9a$ & 9 & $[3,6,27]$ \\
$9I^{0}\mhyphen9c$  & $9$ & $(t^{3} - 6t^{2} + 3t + 1)/(t^{2} - t)$ & $9B^{0}\mhyphen9a$ & 9 & $[1,1,1,6,27]$\\
$9J^{0}\mhyphen9a$  & $9$ & $(t^{3} - 3t + 1)/(t^{2} - t)$ & $9C^{0}\mhyphen9a$ & 9 & $[3,3,3,27]$\\
$9J^{0}\mhyphen9b$  & $9$ & $-18(t^{2} - 1)/(t^{3} - 3t^{2} - 9t + 3)$ & $9C^{0}\mhyphen9a$ & 9 & $[9,9,9,9]$\\
$9J^{0}\mhyphen9c$  & $9$ & $3(t^{3} + 3t^{2} - 9t - 3)/(t^{3} - 3t^{2} - 9t + 3)$ & $9C^{0}\mhyphen9a$ & 9 & $[9,27]$\\
$27A^{0}\mhyphen27a$ & $27$ & $t^{3}$ & $9B^{0}\mhyphen9a$ & 27 & $[27,54,243]$\\
\end{tabular}
}

\bibliographystyle{amsplain}
 \bibliography{bibliography}
 
\end{document}